\documentclass{amsart}

\usepackage{amsmath}
\usepackage{amsfonts}
\usepackage{amssymb}
\usepackage{amsthm}
\usepackage{mathrsfs}
\usepackage{hyperref}
\usepackage{url}
\usepackage{quiver}
\usepackage[T1]{fontenc}

\theoremstyle{plain}
\newtheorem{theorem}{Theorem}[section]
\newtheorem{prop}[theorem]{Proposition}
\newtheorem{lem}[theorem]{Lemma}
\newtheorem{cor}[theorem]{Corollary}

\theoremstyle{definition}
\newtheorem{rem}[theorem]{Remark}
\newtheorem{defn}[theorem]{Definition}
\newtheorem{ex}[theorem]{Example}

\title{Moduli Spaces of Twisted Endomorphisms and Euler Characteristics}
\author{Atharva Korde}
\date{}

\begin{document}

\begin{abstract}
For a (quasi)projective variety $ X $ and an automorphism $ f $ of $ X $, we define moduli spaces of twisted endomorphisms of sheaves on $ X $. If the automorphism $ f $ has finite order, we construct a sheaf of noncommutative algebras $ \mathcal{A}_0(f) $ and show that sheaves with a twisted endomorphism can be described as $ \mathcal{A}_0(f)$-modules. We compute the topological Euler characteristics of Hilbert schemes of points of the algebra $ \mathcal{A}_0(f) $, generalizing the answer for the case where $ f $ is the identity automorphism.
\end{abstract}

\maketitle

\section{Introduction}

\subsection{Overview} 

In this paper, we introduce a notion of `twisted endomorphisms'. These are morphisms $ \phi : \mathcal{F} \to f^* \mathcal{F}$ for sheaves $ \mathcal{F}$ on a scheme equipped with an automorphism $ f$. We study the moduli spaces of pairs $ (\mathcal{F}, \phi) $ of a sheaf with a twisted endomorphism.

Our first main result, proved in Section 2, states that when $ f$ is of finite order, then the category of pairs $ (\mathcal{F}, \phi)$ is essentially the same as modules over a certain sheaf of noncommutative algebras.
\begin{theorem}
Let $ X$ be a quasi-projective variety and $ f : X \to X$ be an automorphism of finite order. Denote by $ \mathsf{(Q)Coh}_{X,f} $ the category whose objects are pairs $ (\mathcal{F}, \phi)$ where $ \mathcal{F}$ is a (quasi)coherent sheaf on $ X$ and $ \phi : \mathcal{F} \to f^* \mathcal{F}$ is a twisted endomorphism, and whose morphisms $ (\mathcal{F}, \phi) \to (\mathcal{F}', \phi')$ are morphisms $ \eta : \mathcal{F} \to \mathcal{F}'$ such that $ (f^*\eta) \phi = \phi \eta$. Then, there is a quasi-coherent sheaf of noncommutative algebras $ \mathcal{B}_0(f)$ on the quotient $ X/f $ and a coherent sheaf of noncommutative algebras $ \mathcal{A}_0(f)$ on $ X/f \times \mathbb{A}^1$ such that: \begin{itemize}
    \item The category $ \mathsf{(Q)Coh}(\mathcal{B}_0(f))$ of (quasi)coherent $ \mathcal{B}_0(f) $-modules and $ \mathsf{(Q)Coh}_{X,f} $ are equivalent.
    \item The category $ \mathsf{QCoh}(\mathcal{A}_0(f))$ of quasicoherent $ \mathcal{A}_0(f) $-modules is equivalent to $ \mathsf{QCoh}_{X,f} $. This equivalence restricts to an equivalence between $ \mathsf{Coh}(\mathcal{A}_0(f))^{pr} $ - the category of coherent $ \mathcal{A}_0(f)$-modules whose support is proper over $ X/f$, and $ \mathsf{Coh}_{X,f}  $.
\end{itemize} 
\end{theorem}

Next, we specialize to coherent sheaves of fixed length $ n \ge 0$. We prove that the stack $ \mathcal{P}^n(X,f) $ parametrizing pairs $ (\mathcal{F}, \phi) $ of coherent sheaves of length $ n$ and a twisted endomorphism is algebraic. Adding a section $ s \in H^0(X, \mathcal{F})$, we consider the moduli space $ \mathscr{H}^n_{X,f} $ parametrizing triples $ (\mathcal{F}, \phi, s) $ such that $ s$ and $ \phi$ together generate $ \mathcal{F}$. It is shown that this moduli space is the Hilbert scheme parametrizing quotients of $ \mathcal{A}_0(f)$.

For any integer $ m \ge 0$, let $ p_3(m) $ be the number of ways to create a `3D-partition' of $ m$. This is the number of ways to stack unit cubes in the positive octant, starting from the corner (in analogy with Young diagrams which yield `2D-partitions'). The generating series $ \sum_{m \ge 0} p_3(m)q^m $ is well-known to be equal to the infinite product $ \prod_{k \ge 1} \frac{1}{(1-q^k)^k}$. The function $ M(q) = \prod_{k \ge 1} \frac{1}{(1-q^k)^k}$ is the MacMahon function.
\par
Let $ e(\cdot) $ denote the topological Euler characteristic. Our second result, proved in Section 3, gives the Euler characteristics $ e(\mathscr{H}^n_{X,f})$.

\begin{theorem}
    Let $ X$ be a smooth projective surface over $ \mathbb{C}$ and $ f : X \to X$ be an automorphism of finite order. Then the generating function for the Euler characteristics of $ \mathscr{H}^n_{X,f}$ satisfies the identity \[ \sum_{n \ge 0} e(\mathscr{H}^n_{X,f})q^n = M(q)^{e(X)} \] 
\end{theorem}

The right hand side is independent of the automorphism. If $ f = id_X $ is the trivial automorphism, then the moduli spaces $ \mathscr{H}^n_{X, id_X} $ are the Hilbert schemes $ \operatorname{Hilb}^n(X \times \mathbb{A}^1) $. So Theorem 1.2 is an extension of the result of Cheah in \cite{cheah96}, which says that for any smooth threefold $ Y$, $ \sum_{n \ge0} e(\operatorname{Hilb}^n(Y))q^n = M(q)^{e(Y)}$.

\subsection{Related Work}

Twisted endomorphisms of the form $ \mathcal{F} \to \mathcal{F} \otimes \mathcal{L} $ a sheaf $ \mathcal{F} $ and a line bundle $ \mathcal{L} $ on $ X $ were considered by Beauville-Narasimhan-Ramanan in \cite{beauville1989spectral}. In that case, pairs $ (\mathcal{F}, \phi: \mathcal{F} \to \mathcal{F} \otimes \mathcal{L}) $ are equivalent to sheaves on the total space $ \operatorname{Tot}(\mathcal{L})$. Since $ e(\operatorname{Tot}(\mathcal{L})) = e(X) $, the generating function for $ e(\operatorname{Hilb}^n(\operatorname{Tot}(\mathcal{L})))$ is $ M(q)^{e(X)} $ for any smooth projective surface $ X$, by \cite{cheah96}. We take the different direction of twisting by an automorphism, and show that this also produces the `correct' identity via Theorem 1.2. 
\par 
The generating series for Euler characteristics of Hilbert schemes of points on a smooth surface was found by Göttsche in \cite{gottsche2001motive}. For singular surfaces with Kleinian singularities, Gyenge-Némethi-Szendröi conjecture and compute the generating series in \cite{gyenge2018euler}. The enumerative theory for $G $-invariant Hilbert schemes has been studied by Bryan-Gyenge in \cite{bryan2022euler} and Pietromonaco in \cite{pietromonaco2021g} for K3 and Abelian surfaces equipped with the action of a finite symplectic group $ G$, and it is shown that the generating series are related to modular forms.
\par
The MacMahon function is ubiquitous in enumerative geometry and Donaldson-Thomas theory. For example, the generating series for Donaldson-Thomas invariants of Hilbert schemes of points on threefolds are expressed in terms of $ M(q)$ and invariants of $ X$ by Behrend-Fantechi in \cite{BehrendFantechi2008}, Levine-Pandharipande in \cite{levine2009algebraic} and Li in \cite{Li2006Zero}. 
\par
The higher-rank Donaldson-Thomas theory was considered by Fasola-Monavari-Ricolfi in \cite{fasola2021higher}. The torus fixed points of the `local-Quot scheme' $ \operatorname{Quot}_{\mathbb{A}^3}(\mathcal{O}^{\oplus r}, n)$ parameterizing length $ n$ quotients of the trivial bundle of rank $ r$, are proven to be in bijection with $ r$-colored plane partitions (Proposition 3.12 of \cite{fasola2021higher}). Our result in Proposition 3.5 related to the local model in the noncommutative setting may be regarded as a direct analogue.

\subsection{Acknowledgments}

I would like to thank my PhD advisor Kai Behrend for introducing me to this line of research, and for many helpful discussions. Parts of this work have appeared in my PhD thesis.

\subsection{Conventions}

 Unless otherwise stated, a product $ S \times T$ of schemes is always assumed to be over a base field. If $ g : S \to T$ is a morphism of schemes and $ U$ is any scheme, then $ g_U$ denotes the morphism $ g \times id_U : S \times U \to T \times U$. The projections $ S \times T \to S$ and $ S \times T \to T$ are denoted by $ p_S $ and $ p_T$ respectively. The notation $ |U|$ is used for the underlying topological space of a scheme $ U$.
\par
For a scheme $ Z$ of finite type over the complex numbers, the (topological) Euler characteristic is $ e(Z) = \sum_{i \ge 0} (-1)^i \dim H^i(Z^{an}, \mathbb{C})$, where $ Z^{an} = Z(\mathbb{C})$ with the analytic topology. It satisfies the excision property: For a finite-type scheme $ Z$, $ e(Z) = e(W) + e(Z - W) $ for any closed subscheme $ W \subset Z$. It also satisfies the multiplicative property: For finite-type schemes $ E,B,F $, we have $ e(E) = e(B)e(F)$ where $ E \to B $ is a locally trivial fibration in the analytic topology with fiber $ F$.

\subsection{Statement about the use of AI} Generative AI was not used to obtain any of the results in this paper.

\section{The Moduli Spaces}

\newcommand{\f}{\mathcal{F}}
\newcommand{\ff}{\mathscr{F}}
\newcommand{\suppf}{\operatorname{Supp} \mathscr{F}}
\newcommand{\cohnx}{\mathcal{C}\! \mathit{oh}^n(X)}
\newcommand{\cohn}{\mathcal{C}\! \mathit{oh}^n}
\newcommand{\homm}{\mathcal{H}\! \mathit{om}}
\newcommand{\pnxf}{\mathcal{P}^n(X,f)}
\newcommand{\rzt}{R[z; \theta]}
\newcommand{\bo}{\mathcal{B}_0(f)}
\newcommand{\ao}{\mathcal{A}_0(f)}
\newcommand{\af}{\mathcal{A}(f)}
\newcommand{\hnxf}{\mathscr{H}^n_{X,f}}

\subsection{Some Technical Background} 

Let $ X $ be a quasi-projective variety over an algebraically closed field $ k $ of characteristic 0. We review some definitions and results on families of coherent sheaves. 
\par
Let $ n \ge 0 $ be a fixed non-negative integer and $ T $ be a $ k$-scheme. A family of 0-dimensional coherent sheaves of length $ n $ parameterized by the base $ T$ is a coherent sheaf $ \ff $ on $ X \times T$ such that \begin{itemize}
    \item $ \ff$ is flat over $ T$.
    \item The support $ \suppf $ is finite over $ T$.
    \item $ (p_T)_* \ff $ is a vector bundle of rank $ n$ on $T$.
\end{itemize}
The stack $ \cohnx $ is the stack parameterizing 0-dimensional coherent sheaves of length $ n$ on $ X$. More precisely, this is the category whose objects are pairs $ (T, \ff) $ where $ T $ is a $ k $-scheme and $ \ff$ is a family of 0-dimensional coherent sheaves of length $ n$ parameterized by $ T$, and morphisms from $ (T', \ff')$ to $ (T, \ff)$ are pairs $ (h, \psi)$ where $ h : T' \to T$ is a $k$-morphism and $ \psi : h_X^* \ff \to \ff' $ is an isomorphism. The fact that $ \cohnx $ is an algebraic stack is a well-known result proved in many sources; see, for example, \cite{fantechi2024stack0dimensionalcoherentsheaves} or \cite{fantechi2025stack0dimensionalcoherentsheaves} for some special cases or \cite[\href{https://stacks.math.columbia.edu/tag/08WC}{Theorem 08WC}]{stacks-project} for a more general statement.
\vspace{6pt}
\par
Next, we aim to construct a moduli space parameterizing sheaves with a twisted endomorphism. Let $ f : X \to X$ be a $k$-automorphism.

\begin{defn} 
    Let $ \f $ be a coherent sheaf on $ X$. A twisted endomorphism of $ \f$ is a morphism $ \phi : \f \to f^* \f$.
\end{defn}

Note that $ \phi $ is not a morphism of the sheaf $ \f$ but can instead be viewed as lying over the automorphism $ f $. However, $ \phi$ itself needn't be an isomorphism.

\begin{defn}
    The category $ \pnxf$ is defined as follows. The objects are triples $ (T, \ff, \Phi) $ where $ (T, \ff) $ is an object of $ \cohnx $ and $ \Phi : \ff \to f_T^* \ff$ is a morphism. A morphism from $ (T', \ff', \Phi') $ to $ (T, \ff, \Phi)$ is a morphism $ (h, \psi): (T', \ff') \to (T,\ff) $ in $ \cohnx$ such that the square
\[\begin{tikzcd}
	{h_X^*\ff} & {h_X^* f_T^* \ff } \cong f_{T'}^*h_X^* \ff \\
	{\ff'} & {f_{T'}^* \ff'}
	\arrow["{h_X^* \Phi}", from=1-1, to=1-2]
	\arrow["{\psi}"', from=1-1, to=2-1]
	\arrow["{f_{T'}^* \psi}", from=1-2, to=2-2]
	\arrow["\Phi"', from=2-1, to=2-2]
\end{tikzcd}\]
commutes.
\end{defn}

By forgetting the twisted endomorphism, there is a natural morphism $ \pnxf \to \cohnx $. The fact that this morphism is representable by algebraic spaces may be obtained as a direct consequence of Proposition 2.1.3 of \cite{lieblich2006remarksstackcoherentalgebras}. For brevity, let $ \mathcal{S} := \cohnx $. The automorphism $ f $ of $ X $ induces a natural  automorphism $ f_{\mathcal{S}} $ of $ \mathcal{S}$. Let $ \mathscr{U} $ be the universal coherent sheaf on $ X \times \mathcal{S}$. Note that the Hom-stack $ \underline{\homm}(\mathscr{U}, f_{\mathcal{S}}^* \mathscr{U}) $ is precisely $ \pnxf$. Since $ f_{\mathcal{S}}^* \mathscr{U} $ is $ \mathcal{S} $-flat and has proper (in fact finite) support over $ \mathcal{S}$, Proposition 2.1.3 of \cite{lieblich2006remarksstackcoherentalgebras} applies by taking $ Y = X \times \mathcal{S} $ and $ S = \mathcal{S}$. It follows that $ \pnxf $ is also an algebraic stack.

\begin{lem}
    Let $ (T, \ff, \Phi)$ be an object in $ \pnxf$. For $ t \in T$, let $ \ff_t$ denote the sheaf $ \ff|_{X \times t}$ and $ \Phi_t : \ff_t \to f^* \ff_t $ be the morphism $ \Phi $ restricted to $ X \times t $. The locus of points $ t \in T $ such that $ \Phi_t $ is an isomorphism is open in $ T$, and the locus of points $ t \in T $ for which $ \Phi_t = 0$ is closed in $ T $.
\end{lem}

\begin{proof}
    Let $ \mathscr{L}$ be the cokernel of $ \Phi$. Since $ \operatorname{Supp} \mathscr{L}$ is finite over $ T$, the image $ p_T(|\operatorname{Supp} \mathscr{L}|) \subset T $ is closed. Let $ U$ be the open complement. Then $ \Phi_t$ is surjective iff $ t \in U$, so we may assume $ U=T$ and that $ \Phi$ is surjective. Let $ \mathscr{K} $ be the kernel of $ \Phi$. Since $ \ff$ and $ f^* \ff$ are $ T$-flat, $ \mathscr{K}_t$ is the kernel of $ \Phi_t$ for every $ t \in T$ (by Lemma 2.1.4 in \cite{huybrechts2010geometry}). Since $ \operatorname{Supp} \mathscr{K}$ is also finite over $ T$, the previous argument applies and we get an open subset $ V \subset T$ such that $ \Phi_t$ is an isomorphism for all $ t \in V$.
    \par
    That the fiber-wise zero locus $ \{ t \in T \: | \: \Phi_t=0 \}$ is closed follows from a direct application of Corollaire 7.7.8 in \cite{EGAIII_2}.
\end{proof}

From the previous lemma, we have an open substack $ \pnxf^{\times} \subset \pnxf $ parameterizing pairs $ (\ff, \phi) $ where $ \phi$ is an isomorphism. There is an alternative description of $ \pnxf $. Let $ \mathcal{S} = \cohnx$ with the natural automorphism $ f_{\mathcal{S}} $ as before. Then we have a 2-cartesian square of stacks
\[\begin{tikzcd}
	{\pnxf^{\times}} & {\mathcal{S}} \\
	{\mathcal{S}} & {\mathcal{S} \times \mathcal{S}}
	\arrow[from=1-1, to=1-2]
	\arrow[from=1-1, to=2-1]
	\arrow["{\Gamma_{f_{\mathcal{S}}}}", from=1-2, to=2-2]
	\arrow["\Delta"', from=2-1, to=2-2]
\end{tikzcd}\]
where $ \Gamma_{f_{\mathcal{S}}}$ is the graph of $ f_{\mathcal{S}} $ and $ \Delta $ is the diagonal morphism.

\begin{rem}
    When $ f = id_X $, the stack $ \pnxf $ is isomorphic to $ \cohn(X \times \mathbb{A}^1)$. If $ \phi : \f \to \f $ is an endomorphism, $ \f$ gets the structure of a coherent $ \mathcal{O}_X[t]$-module through $ \phi$ and, therefore, defines a 0-dimensional coherent sheaf $ \mathcal{G}$ on $ \underline{\operatorname{Spec}}(\mathcal{O}_X[t]) = X \times \mathbb{A}^1 $ such that $ (p_X)_* \mathcal{G} = \mathcal{F} $. Conversely, if $ \mathcal{G} $ is a 0-dimensional coherent sheaf on $ X \times \mathbb{A}^1 $, then $ \f = (p_X)_* \mathcal{G}$ is a coherent $ \mathcal{O}_X[t] $-module on $ X $ as $ \operatorname{Supp} \mathcal{G} $ is finite over $ X $. Then $ \f $ has an endomorphism defined by the action of $ t $. The lengths of $ \f$ and $ \mathcal{G}$ are equal because $ h^0(X \times\mathbb{A}^1, \mathcal{G}) = h^0(X, (p_X)_* \mathcal{G}) = h^0(X, \mathcal{F}) $.
    \par
    It is a routine check to verify that this correspondence carries over to families of sheaves and defines an isomorphism $ \pnxf \cong \cohn(X \times \mathbb{A}^1) $.
\end{rem}

\subsection{The Case of Finite Order} From now on, assume that $ f $ has finite order, say $ d$. We first discuss some preliminaries in the affine case.

\begin{defn} (Twisted Polynomial/Laurent Rings)
    Let $ R $ be a ring (not necessarily commutative) and $ \theta : R \to R $ be a ring automorphism. The twisted polynomial ring $ R[z; \theta]$ has the same elements and additive structure as the ordinary polynomial ring $ R[z] $. The multiplication of two monomials $ r_1z^{n_1} $ and $ r_2z^{n_2} $ ($ r_1,r_2 \in R, \: n_1, n_2 \ge 0 $) in $ R[z; \theta] $ is defined by the rule $$ (r_1 z^{n_1}) (r_2z^{n_2}) = r_1 \theta^{n_1}(r_2) z^{n_1+n_2}$$ and extended linearly. The twisted Laurent ring $ R[z, z^{-1}; \theta] $ is defined similarly - the multiplication rule holds for all $ n_1, n_2 \in \mathbb{Z}$.
\end{defn}

\begin{lem}
    Let $ R$ be a commutative ring and $ \theta$ have finite order $ d $. Let $ R^{\theta} := \{ r \in R \: | \: \theta(r) = r \} $ be the ring of invariants. Then, the center of $ \rzt $ contains $ R^{\theta}[z^d] $ and the center of $ R[z, z^{-1}; \theta] $ contains $ R^{\theta}[z^d, z^{-d}]$. If $ R$ is a domain, these containments are equalities.
\end{lem}

\begin{proof}
 We prove the polynomial case, the proof for Laurent polynomials is the same. For any $ k \ge 0 $, it is clear that $ rz^{kd} $ is a central element whenever $ r \in R^{\theta} $ because for any monomial $ sz^n$, $ (rz^{kd})(sz^n) = r\theta^{kd}(s)z^{kd+n} = rsz^{kd+n} $ and $ (sz^n)(rz^{kd}) = s\theta^n(r) z^{n+kd} = sr z^{n+kd} $. Hence, $ R^{\theta}[z^d]$ is a central subring of $ \rzt$.
 \par
 Conversely, assume that $ p(z) = \sum_i r_iz^i$ is central in $ \rzt$. Since $ p(z) \cdot z = z \cdot p(z) $, we get $ \sum_i (r_i - \theta(r_i))z^{i+1} = 0 $ so that $ r_i \in R^{\theta} $ for all $ i $. Since $ p(z) \cdot r = r \cdot p(z) $ for every $ r \in R$, we get $ \sum_i r_i(\theta^i(r)-r)z^i = 0 $ so that $ r_i(\theta^i(r)-r) = 0$ for every $ i \ge 0 $ and every $ r \in R$. If $ d $ divides $ i $, this equation is true for any $ r$. If $ d$ does not divide $ i$, $ \theta^i$ is not the identity automorphism and choosing an $ r$ such that $ \theta^i(r) \neq r$, we see that $ r_i = 0$. It follows that $ p(z) $ is of the form $ \sum_{j} r_{dj}(z^d)^j$ with $ r_{dj} \in R^{\theta}$, so the center is exactly $ R^{\theta}[z^d] $.
\end{proof}

\begin{cor}
    Let $ R$ be a finitely generated commutative $ k$-algebra which is a domain, and let $ \theta$ have finite order $ d$. Then $ \widetilde{\rzt}$ is a coherent sheaf of noncommutative algebras on $ \operatorname{Spec} R^{\theta}[z^d] $.
\end{cor}

\begin{proof}
    The ring $ R^{\theta}[z^d]$ is Noetherian. In addition, $ R$ is integral over $ R^{\theta} $ and hence is finitely generated as a $ R^{\theta}$-module. So $ \rzt$ is a finitely generated $ R^{\theta}[z^d]$-module as can be seen from the tower of modules $ R^{\theta}[z^d] \subset R[z^d] \subset \rzt $.
\end{proof}

\begin{rem} (Functoriality)
    Let $ \alpha : R \to R'$ be a homomorphism of commutative rings. Let $ \theta $ and $ \theta'$ be automorphisms of $ R $ and $ R'$ of order $ d$, such that $ \theta' \alpha = \alpha \theta$. For brevity, denote by $ I $ and $ I' $ the invariants $ R^{\theta} $ and $ (R')^{\theta'} $.
    \par
    There is a natural induced homomorphism $ \rzt \to R'[z; \theta']$ that maps $ I[z^d] $ into $ I'[z^d] $. Moreover, the natural map $ R \otimes_I I' \to R' $ is compatible with the automorphisms. That is, the square
    \[\begin{tikzcd}
	{R \otimes_I I'} & {R'} \\
	{R \otimes_I I'} & {R'}
	\arrow[from=1-1, to=1-2]
	\arrow["{\theta \otimes id_{I'}}"', from=1-1, to=2-1]
	\arrow["{\theta'}", from=1-2, to=2-2]
	\arrow[from=2-1, to=2-2]
\end{tikzcd}\]
    is commutative. From this it follows that if $ m : R[z] \otimes_I R[z] \to R[z]$ denotes the multiplication map of $ \rzt$ and likewise for $ m'$, then multiplication is compatible under passage to invariants. That is,
    \[\begin{tikzcd}
	{(R[z] \otimes_I R[z]) \otimes_I I'} && {R[z] \otimes_I I'} \\
	{R'[z] \otimes_{I'} R'[z]} && {R'[z]}
	\arrow["{m \otimes id_{I'}}", from=1-1, to=1-3]
	\arrow[from=1-1, to=2-1]
	\arrow[from=1-3, to=2-3]
	\arrow["{m'}"', from=2-1, to=2-3]
\end{tikzcd}\]
    is also a commutative square.
\end{rem}

These local considerations allow us to define a noncommutative sheaf of algebras on the quotient variety $ X/f $. Recall that the quotient $ X/f $ is constructed as follows. Cover $ X $ by affine open $ f $-invariant subschemes $ \{ U_i = \operatorname{Spec} R_i \}_i $. Let $ U_{ij} = \operatorname{Spec} R_{ij} = U_i \cap U_j $ be the intersections which are also affine (as $ X$ is separated) and $ f$-invariant. Each of the rings $ R_i$, respectively $ R_{ij}$ gets automorphisms $ \theta_i $, respectively $ \theta_{ij} $ that are compatible with $ R_i \to R_{ij}$ in the sense of the previous remark. Then, the quotient is obtained by gluing the spectra of the invariants $ I_i = R_i^{\theta_i} $, that is, $ \operatorname{Spec} I_i$ and $ \operatorname{Spec} I_j$ are glued on $ \operatorname{Spec} I_{ij} $ (and this gluing data is compatible).
\par
Let $ q : X \to X/f $ be the quotient map. To define an (associative) multiplication map $ q_* \mathcal{O}_X[z] \otimes q_* \mathcal{O}_X[z] \to q_* \mathcal{O}_X[z] $, it is sufficient to define multiplications locally on each $ \operatorname{Spec} I_i$ and check compatibility under restriction. On $ I_i$, define $ R_i[z] \otimes_{I_i} R_i[z] \to R_i[z] $ by multiplication in the polynomial ring $ R_i[z; \theta_i^{-1}]$ (note the inverse). Remark 2.8 implies that these local sections of $ \homm_{X/f}(q_* \mathcal{O}_X[z] \otimes q_* \mathcal{O}_X[z], q_* \mathcal{O}_X[z]) $ can be glued to give a well-defined algebra structure. We denote this algebra by $ \bo $. The underlying quasicoherent sheaf of $ \bo $ is $ q_* \mathcal{O}_X[z] $.
\par
The usefulness of $ \bo $ is in the next lemma which says that a (quasi)coherent sheaf with a twisted endomorphism is essentially the same as a (quasi)coherent $ \bo$-module.

\begin{lem}
    Let $ \phi : \f \to f^* \f $ be a twisted endomorphism of a (quasi)coherent sheaf on $ X$. Then $ \mathcal{G} = q_* \f $ is a (quasi)coherent $ \bo$-module. Conversely, a (quasi)coherent $ \bo$-module $ \mathcal{G} $ naturally gives a (quasi)coherent sheaf $ \f$ on $ X $ with a twisted endomorphism $ \phi $, and this correspondence is bijective. Therefore, this correspondence induces an equivalence of categories $ \mathsf{(Q)Coh}_{X,f} $ and $ \mathsf{(Q)Coh}(\bo) $.
\end{lem}

\begin{proof}
    Since $ f^* = (f^{-1})_*$ we have $ q_* f^* \f = q_* \f$, so $ q_* \phi : \mathcal{G} \to \mathcal{G}$ is an endomorphism. Note that $ \mathcal{G} $ is already a $ q_* \mathcal{O}_X$-module. We claim that the action of $ q_*\phi$ gives it the structure of a $ \bo$-module. This is a local question, so assume $ \mathcal{F} \cong \tilde{M}$ on an $ f$-invariant open subset $ \operatorname{Spec} R \subset X $. If $ \theta$ is the induced automorphism of $ R$, then $ \mathcal{G} = \tilde{_I M} $ on $ \operatorname{Spec} I = \operatorname{Spec} R^{\theta}$ (here $ _I M $ is $ M$ viewed as a $ I$-module). On $ \operatorname{Spec} R$, $ \phi: M \to M $ becomes a morphism of abelian groups satisfying $ \phi(rm) = \theta^{-1}(r)\phi(m)$ for every $ r \in R$ and $ m \in M $, or equivalently, $ \phi r = \theta^{-1}(r) \phi $ as endomorphisms of $ M $. So $ \mathcal{G}$ is a $ \bo$-module.
    \par
    Conversely, if $ \mathcal{G}$ is a $ \bo$-module, then it is a $ q_* \mathcal{O}_X$-module and it defines a sheaf $ \f$ on $ X$ such that $ q_* \f = \mathcal{G} $. The same steps as above show that the action of $ \bo$ on $ \mathcal{G}$ gives a twisted endomorphism $ \phi : \f \to f^* \f$. In this correspondence, the morphisms $ (\f, \phi) \to (\f', \phi') $ are morphisms of $ \bo$-modules, and coherent sheaves map to coherent sheaves as $ q $ is finite, so the equivalence of categories is now clear.
\end{proof}

By Lemma 2.6, the center of $ \bo$ is $ \mathcal{O}_{X/f}[z^d]$. Let $ Y_0 = X/f \times \mathbb{A}^1_{z^d} $ and $ p : Y_0 \to X/f$ be the projection. It follows that there is a noncommutative sheaf of algebras $ \ao$ on $ Y_0 $ such that $ p_* \ao = \bo$. Note that $ \ao$ is coherent (Corollary 2.7) while $ \bo$ is only quasicoherent. 
\par
Therefore, there is an equivalence between $ \bo$-modules and $ \ao$-modules and the next corollary is the analogue of Remark 2.4.

\begin{cor}
    Let $ \phi : \f \to f^* \f$ be a twisted endomorphism of a (quasi)coherent sheaf on $ X$. Then, there is a (quasi)coherent $ \ao$-module $ \mathcal{H}$ on $ Y_0 $ such that $ q_* \f = p_* \mathcal{H}$. Conversely, a (quasi)coherent $ \ao$-module $ \mathcal{H}$ naturally gives a (quasi)coherent sheaf $ \f$ on $ X$ with a twisted endomorphism $ \phi$, and this correspondence is bijective. Therefore, there is an equivalence of categories $ \mathsf{QCoh_{X,f}} $ and the category $ \mathsf{QCoh}(\ao) $, which restricts to an equivalence between $ \mathsf{Coh_{X,f}} $ and $ \mathsf{Coh}(\ao)^{pr} $.
\end{cor}

\begin{proof}
    Everything is clear except the last assertion. If $ \mathcal{H}$ has proper support, then $ p_* \mathcal{H}$ is coherent and so is $ \f$. On the other hand, assume that $ \f$ is coherent, and hence so is $ q_* \f $. Checking properness of $ \operatorname{Supp} \mathcal{H}$ over $ X/f$ is a local question, so work locally on $ X/f$ as in Lemma 2.9. Then $ \phi^d : {}_{I}M \to {}_{I}M $ is a morphism of finitely generated $ I$-modules. The determinant trick implies that $ \phi^d $ satisfies a polynomial of the form $ P(t) = t^l + a_{l-1}t^{l-1} + \ldots + a_0 $ with coefficients in $ I$. Then the annihilator of $ {}_{I}M$ in $ I[z^d]$ contains $ P(z^d) = (z^d)^l + a_{l-1}(z^d)^{l-1} + \ldots + a_0$. Therefore, locally over $ \operatorname{Spec} I$, $ \operatorname{Supp} \mathcal{H}$ is contained in the closed subscheme $ \operatorname{Spec} I[z^d]/P(z^d) $ which is finite over $ \operatorname{Spec} I$ as $ P$ is monic. It follows that $ \operatorname{Supp} \mathcal{H}$ is finite and hence proper over $ X/f$ (indeed, being proper and being finite are equivalent in this scenario).
\end{proof}

Once again, it is routine to check that all these constructions carry over to families, since the constructions of $ \bo$ and $ \ao$ are compatible under base change. That is, if $ T$ is a $k$-scheme, then $ \mathcal{B}_0(f_T) $ is the pullback of $ \bo$ under the projection $ (X \times T)/f_T \cong X/f \times T \to X/f $, and similarly, for $ \mathcal{A}_0(f)$. Therefore, we get an alternate description of $ \pnxf$ as the stack of $ \ao$-modules. The equivalence of categories follows

\begin{theorem}
    If $ f $ has finite order, then $ \pnxf $ is isomorphic to the stack $ \cohn(Y_0, \ao)$ parameterizing 0-dimensional coherent $ \ao$-modules of length $ n$ on $ Y_0$. Therefore, $ \pnxf$ is an algebraic stack.
\end{theorem}

\begin{proof}
    We only explain why $ \pnxf \cong \cohn (Y_0, \ao) $ is an algebraic stack with this alternate approach. The stack $ \mathcal{S} = \cohn (Y_0)$ is algebraic. Let $ \ao^{\mathcal{S}} $ be the pullback of $ \ao$ to $ Y_0 \times S$. Let $ \mathscr{U} $ be the universal family, $ m : \ao^{\mathcal{S}} \otimes \ao^{\mathcal{S}} \to \ao^{\mathcal{S}} $ the multiplication, and $ u : \mathcal{O}_{\mathcal{S}} \to \mathcal{A}_0(f)^{\mathcal{S}} $ be the unit. The stack $ \cohn (Y_0, \ao)$ is the vanishing locus of the two maps $$ \underline{\homm}(\ao^{\mathcal{S}} \otimes \mathscr{U}, \mathscr{U}) \to \underline{\homm}(\ao^{\mathcal{S}} \otimes \ao^{\mathcal{S}} \otimes \mathscr{U}, \mathscr{U}) $$ and $$ \underline{\homm}(\ao^{\mathcal{S}} \otimes \mathscr{U}, \mathscr{U}) \to \underline{\homm}(\mathscr{U}, \mathscr{U}) $$ the first map being $ a \to a(1_{\ao^{\mathcal{S}}} \otimes a) - a(m \otimes 1_{\mathscr{U}}) $ and the second map being $ a \to a(u \otimes 1_{\mathscr{U}}) - 1_{\mathscr{U}}$ for any $ a \in \underline{\homm}(\ao^{\mathcal{S}} \otimes \mathscr{U}, \mathscr{U})$, so it is algebraic.
\end{proof}

\subsubsection{Compactification} For future use, we mention that the algebra $ \ao$ can be extended to an algebra $ \af$ on $ Y = X/f \times \mathbb{P}^1_{z^d}$. Let $ w = z^{-1}$. On $ Y_{\infty} := X/f \times \mathbb{A}^1_{w^d}$ we have the algebra $ \mathcal{A}_{\infty}(f) $ obtained using the automorphisms $ \theta_i$ instead of $ \theta_i^{-1}$ in the paragraph before Lemma 2.9 (abstractly, $ \mathcal{A}_{\infty}(f)$ is isomorphic to $ \mathcal{A}_0(f^{-1})$). Then, $ Y_0$ and $ Y_{\infty}$ glue to give $ Y$, and by Remark 2.8, the algebras $ \ao$ and $ \mathcal{A}_{\infty}(f)$ glue on the intersection, where the multiplication locally is simply that of the twisted Laurent rings.

\subsection{Stability and Hilbert Schemes} Since the moduli spaces $ \pnxf $ of pairs $ (\f, \phi) $ in the previous section are stacks, we introduce an additional section $ s \in H^0(X, \f) $. Given a pair $ (\f, \phi) $ and $ s \in H^0(X, \f)$, we get a new section in $ H^0(X, \f) $ since $ \phi s \in H^0(X, f^* \f) \cong H^0(X, \f) $ and the last isomorphism is canonical. By abuse of notation, we denote the new section by $ \phi s$ as well. Similarly, $ \phi^i s $ denotes the $i$-fold application of this operation.

\begin{defn}
    A triple $ (\f, \phi, s) $ is stable if the sections $ \phi^k s , \: k \ge 0 $ generate the sheaf $ \f$.
\end{defn}

\begin{lem}
    The following statements are equivalent for a triple $ (\f, \phi, s)$ with $ \f$ a 0-dimensional sheaf of length $ n$:
    \begin{itemize}
        \item $ (\f, \phi, s) $ is stable.
        \item There is a 0-dimensional $ \bo$-module $ \mathcal{G}$ of length $ n$ on $ X/f$, and a surjection $ \bo \to \mathcal{G}$.
        \item There is a 0-dimensional $ \ao$-module $ \mathcal{H}$ of length $ n$ on $ Y_0$, and a surjection $ \ao \to \mathcal{H}$.
    \end{itemize}
\end{lem}

\begin{proof}
    These equivalences follow immediately from Lemma 2.9 and Corollary 2.10. The condition that the sections generate $ \f$ is equivalent to the morphisms $ \bo \to \mathcal{G} $ and $ \ao \to \mathcal{H}$ induced by $ s$ being surjective.
\end{proof}

Consider the moduli space of stable triples $ (\f, \phi, s)$ where $ \f$ has fixed length $ n$. More formally, let $ \mathscr{H}^n_{X,f}$ be the functor so that for a $ k$-scheme $ T$, $ \hnxf(T)$ consists of equivalence classes of triples $ (\ff, \Phi, s)$ where $ (\ff, \Phi) \in \pnxf(T)$ and the sections $ \Phi^k s$ generate $ \ff$. Two triples $ (\ff, \Phi, s)$ and $ (\ff', \Phi', s')$ are equivalent if there is an isomorphism $ \Psi : \ff \to \ff'$ which makes the obvious diagram commute, and satisfies $ \Psi s = s'$.
\par
By Lemma 2.13, $ \hnxf$ is isomorphic to the Hilbert functor $ \mathscr{H}^n(Y_0, \ao)$ parameterizing quotients $ \ao \to \mathcal{H}$ of $ \ao$-modules, where $ \mathcal{H}$ is 0-dimensional of length $ n$. The existence of Hilbert schemes of noncommutative algebras has been addressed in the context of Simpson's work on sheaves of differential operators in \cite{simpson1994moduli}. However, we give a short proof here to elucidate some key points.

\begin{lem}
    The Hilbert functor $ \mathscr{H}^n(Y_0, \ao) $ is representable by a quasi-projective scheme $ \operatorname{Hilb}^n(Y_0, \ao)$.
\end{lem}

\begin{proof}
    First assume that $ X $ is projective. Consider the compactification $ Y$ and the algebra $ \af$ and denote by $ Q = \operatorname{Quot}^n(Y, \af)$ the Quot scheme that parametrizes 0-dimensional quotients of length n of the underlying $ \mathcal{O}_Y$-module $ \af$. Let $ 0 \to \mathscr{K} \to \af^Q \to \mathscr{E} \to 0$ be the universal quotient on $ Q \times Y$. A quotient $ x = [\af \to \mathcal{E} \to 0] \in Q $ is a quotient of $ \af$-modules iff the kernel $ \mathcal{K}$ is $ \af $-invariant. Therefore, the locus of points $ x \in Q$ which are surjections of $ \af$-modules is the fiber-wise zero locus of the composition $$ \nu : \af^Q \otimes_{\mathcal{O}_Y} \mathscr{K} \to \af^Q \otimes_{\mathcal{O}_Y} \af^Q \to \af^Q \to \mathscr{E} $$ 
    Since $ \mathscr{K}, \mathscr{E}, $ and $ \af^Q $ are all $ Q $-flat and $ Q \times Y \to Q$ is projective, Corollaire 7.7.8 in \cite{EGAIII_2} applies, and we get a coherent sheaf $ \mathcal{N}$ on $ Q$ such that $$ \operatorname{Hom}_{\mathcal{O}_Q}(\mathcal{N}, \mathcal{O}_Q) = \operatorname{Hom}_{\mathcal{O}_{Q \times Y}}(\mathscr{K} \otimes_{\mathcal{O}_Y} \af^Q , \mathscr{E}) $$ The map $ \nu $ corresponds to a map $ n \in \operatorname{Hom}_{\mathcal{O}_Q}(\mathcal{N}, \mathcal{O}_Q) $. The image of $ n$ defines an ideal sheaf, and the closed subscheme defined by this ideal sheaf is exactly the locus of points $ x \in Q$ such that $ \nu_x = 0$. This closed subscheme $ \operatorname{Hilb}^n(Y, \af) $ represents the functor $ \mathscr{H}^n(Y, \af)$. Then $ \mathscr{H}^n(Y_0, \ao) $ is representable by an open subscheme of this closed subscheme - it is the inverse image of $ \operatorname{Sym}^nY_0$ under the Hilbert-Chow morphism $ \mathfrak{c}_n : \operatorname{Hilb}^n(Y, \af) \to \cohn(Y) \to \operatorname{Sym}^n Y $.
    \par
    If $ X $ is quasi-projective, we can first pass to an equivariant compactification $ (\overline{X}, \overline{f})$. That is, there is a projective variety $ \overline{X}$ and an automorphism $ \overline{f} : \overline{X} \to \overline{X}$ such that $ \overline{f}|_{X} = f$. This can be easily obtained by embedding $ X$ in some projective space $ \mathbb{P} $ using a $ f $-invariant very ample line bundle of the form $ \otimes_{i=0}^{d-1} (f^i)^*\mathcal{L} $ for some very ample line bundle on $ X$, and then taking the compactification inside $ \mathbb{P}$. The Hilbert functor is representable for $ (\overline{X}, \overline{f})$, and we then restrict to sheaves supported on $ X$ exactly as in the last line of the previous paragraph. This completes the proof of representability in the general case.
\end{proof}

\begin{rem}
    In the previous lemma, $ \mathcal{N}$ is only a coherent sheaf in general and $ Q$ may not be smooth, so the construction does not provide a (global) Kuranishi model, i.e. $ \operatorname{Hilb}^n(Y, \af)$ is not directly seen to be the zero-locus of a section of a vector bundle on a smooth scheme.
\end{rem}

By abuse of notation, let $ \hnxf$ also denote the scheme representing the functor described by either of the three equivalent conditions in Lemma 2.13.

\begin{rem}
    The locus of triples in $ \hnxf$ in which $ \phi = 0$ is precisely the Hilbert scheme of the surface $ \operatorname{Hilb}^n(X)$. It is a closed subscheme of $ \hnxf$ by Lemma 2.3. In general, $ \hnxf$ is larger than $ \operatorname{Hilb}^n(X)$.
\end{rem}

\begin{ex}
    We give a description of $ \mathscr{H}^1_{X,f}$. For a $ T $-family of stable triples $ (\ff, \Phi, s)$, the support $ \operatorname{Supp} \ff$ is a section of the projection $ p_T : X \times T \to T$. Let $ i : T \to X \times T $ denote the section whose image is $ |\operatorname{Supp} \ff| $. Then $ (p_T)_* \ff = i^* \ff$ is a line bundle on $ T$. The section $ s$ defines a trivialization of $ i^* \ff $ as $ s$ cannot vanish at any point on the support. Since $ \operatorname{Hom}(\ff, (f_X)^* \ff) \cong \operatorname{Hom}(i^* \ff, i^{!} (f_X)^* \ff) $, it follows that such a triple is determined by a map $ j : T \to X $, an isomorphism $ s : \mathcal{O}_T \to \mathcal{O}_T $, and a map $ \mathcal{O}_T \to \mathcal{Q} $ for a sheaf $ \mathcal{Q}$ (determined by $ j$) which is supported only on those points $ t \in T$ such that $ j(t)$ is fixed by $ f$. Therefore, at the level of $ \mathbb{C}$-points, $ \mathscr{H}^1_{X,f}$ is the union of $ X$ (the locus $ \phi=0$) and an $ \mathbb{A}^1$-bundle over the fixed locus of $ f$ which meets $ X$ transversally.
\end{ex}

\begin{rem}
    In the spirit of \cite{Hartshorne1966}, we conjecture that $ \hnxf$ is connected for all $ n \ge 1$. The statement is not true for the compactification $ \operatorname{Hilb}^n(Y, \af)$, as can be seen by considering $ \operatorname{Hilb}^1(Y, \af)$ for some $ f$ without fixed points - in this case, $ \operatorname{Hilb}^1(Y, \af)$ is a disjoint union of two copies of $ X$. However, we also suspect that this is the only kind of example in which connectedness fails to hold for the compactification.
\end{rem}

\begin{rem}
    On $ \pnxf$ and $ \hnxf $, there is a natural $ \mathbb{G}_m$-action that scales the twisted endomorphism $ \phi$. If $ \phi$ is an isomorphism, then $ (\f, \phi, s) \in \hnxf $ cannot be a $ \mathbb{G}_m$-fixed point. Indeed, if it were, this would imply that for all $ \lambda $, $ \det(\lambda \phi) = \det(\psi^{-1}) \det(\phi) \det(\psi) $ for linear endomorphisms $ \phi $ and $ \psi $ of the finite dimensional vector space $ H^0(X, \f)$. This is impossible if $ \det(\phi) \neq 0$. It follows that if a quotient $ \af \to \mathcal{H} \to 0 $ is supported in $ X/f \times \mathbb{G}_m$, then it is not a fixed point of this $ \mathbb{G}_m$-action because such a quotient corresponds (by Lemma 2.13) to a triple $ (\f, \phi, s)$ in which $ \phi$ is an isomorphism.
\end{rem}

\section{Euler Characteristics}

\newcommand{\cx}{\mathbb{C}}
\newcommand{\mna}{\mathcal{M}^n_A}

In this section, $ X$ is a smooth projective surface over $ k = \cx $. The automorphism $ f$ remains of finite order $ d$. We prove the formula for the generating function for the Euler characteristics of the Hilbert schemes in Section 2.

\subsection{Hilbert Schemes of Associative Algebras} In \cite{larsen2012hilbertschemespointsassociative}, Larsen and Lunts have constructed Hilbert schemes of points of associative algebras. We start with a brief description of their construction.
\par
Let $ A$ be a finitely generated associative algebra over $ \cx $. For a $ \cx$-scheme $ T$, let $ A^T$ denote the sheaf of algebras associated to $ A \otimes_{\cx} \mathcal{O}_T$. The Hilbert scheme of left ideals of codimension $ n$ of $ A$ parametrizes $ A $-module quotients $ A \to Q \to 0$ where $ \dim_{\cx} Q = n $.

\begin{defn}
    (The Hilbert functor of left ideals of codimension $ n$) The contravariant functor $ \mathcal{M}^n_A $ from the category of schemes over $ \cx$ to sets is defined as follows. For a $ \cx$-scheme $ T$, $ \mna(T)$ is the set of equivalence classes of pairs $ (M,v)$ such that
    \begin{itemize}
        \item M is a left $ A^T$-module.
        \item As an $ \mathcal{O}_T$-module, $ M$ is locally free of rank $ n$.
        \item The section $ v \in H^0(T,M)$ generates $ M$ as an $ A^T$-module.
    \end{itemize}
    Two pairs $ (M,v)$ and $ (M',v')$ are equivalent if there is an isomorphism $ \psi : M \to M'$ of $ A^T$-modules such that $ \psi(v)=v'$. For a morphism $ T' \to T$ of schemes, the map $ \mna(T) \to \mna(T') $ is the obvious map of pullback of sheaves.
\end{defn}

\begin{prop}
    (Larsen, Lunts) The functor $ \mna$ is represented by a finite type $ \cx$-scheme $ H^n_{A}$.
\end{prop}

%\begin{rem}
    %(Relation to McKay quivers) Let $ G$ be a finite group and $ \rho_1, \rho_2, \ldots, \rho_l$ be its irreducible representations. If $ V$ is a representation of $ G$, the McKay quiver associated to $ (G,V)$ is defined as follows: The vertices of the quiver are the irreducible representations and the number of edges from $ \rho_i$ to $ \rho_j$ is the multiplicity of $ \rho_j$ in $ V \otimes \rho_i$.
    %\par
    %If $ G$ is abelian, the irreducible representations are one-dimensional, arising from characters. Let $ G \cong \widehat{G}$ be a non-canonical isomorphism between $ G$ and its group of characters. For every irreducible representation $ \rho_i $, let $ \chi_i $ be the associated character and $ f_i \in G$ the inverse image under the isomorphism. If $ S = (f_{i_1}, f_{i_2}, \ldots, f_{i_r})$, let $ V = \sum_{k=1}^{r} \rho_{i_k}$. Then, in the McKay quiver associated to $ (G,V)$, the number of edges from $ \rho_i$ to $ \rho_j$ is the same as the number of $ k$ such that $ f_j = f_{i_k}f_i$. So, ignoring the relations, the Cayley quiver of $ (G,S)$ can be non-canonically identified with the McKay quiver of $ (G,V)$.
%\end{rem}

Let $ c \ge 1 $ be any integer, and let $ D = \prod_{i=0}^{c-1} \cx e_i$. Consider a non-commutative ring of the form \[ B = D \langle X_1,X_2,X_3 \rangle /(X_1X_2 - \alpha X_2X_1, X_2X_3 - \beta X_3X_2, X_3X_1 - \gamma X_1X_3) \] where $ \alpha, \beta, \gamma \in D^{\times} $ and for every $ i $, the following relations hold. \begin{itemize}
    \item $ e_iX_1 = X_1e_i $
    \item $ e_iX_2 = X_2e_i $
    \item $ e_{i+1}X_3 = X_3 e_i $
\end{itemize} We assume that $ B$ is an iterated twisted polynomial ring starting from $ D$, and by successively adjoining indeterminates $ X_1, X_2, X_3$ twisted by some automorphisms (the twist applied when adjoining $ X_1$ is the identity automorphism because of the first relation $ e_iX_1 = X_1e_i $). In that case, every $ b \in B $ can be written uniquely as a sum of monomials of the form $ e_iX_1^{m_1}X_2^{m_2}X_3^{m_3} $, that is, $ X_1,X_2,X_3$ can be rearranged to appear in lexicographical order. We will need to compute the Euler characteristics of the Hilbert schemes $ H^n_{B}$. The following proposition of Białynicki-Birula in \cite{BialynickiBirula1973OnFP} will be useful.

\begin{prop}
    Let $ X$ be a finite-type $ \cx $-scheme with an action of the multiplicative group $ \mathbb{G}_m$. Let $ X^{\mathbb{G}_m}$ be the scheme of fixed-points of the action. Then $ e(X) = e(X^{\mathbb{G}_m})$.
\end{prop}

\begin{defn}
    A left ideal $ I \subset B $ is monomial if for any $ b \in I $, every monomial in the (unique) representation of $ b$ as a sum $ \sum e_iX_1^{m_1}X_2^{m_2}X_3^{m_3} $ is also in $ I$.
\end{defn}

\begin{prop}
    The Euler characteristic of $ H^n_B $ is given by the formula $$ e(H^n_B) = \sum_{\sum_{i=0}^{c-1} n_i = n  ,n_i \ge 0} \left( \prod_{i=0}^{c-1} p_3(n_i) \right)$$ Consequently, the generating series $ \sum_{n \ge 0} e(H^n_B)q^n $ equals $ M(q)^{c}$.
\end{prop}

\begin{proof}
    Consider the $ \mathbb{G}_m$-action on $ B$ which scales $ X_3$. It induces a natural action on the $ \cx $-points of $ H^n_B $, or equivalently on the left ideals of $ B $ of codimension $ n$. Let $ I \subset B$ be a $ \mathbb{G}_m$-invariant left ideal. For $ b \in I $, write $ b = \sum_{i =0}^{c-1} e_ib_i(X_1, X_2, X_3)$ where $ b_i(X_1, X_2, X_3) $ denotes the sum of all the monomials which have a coefficient of $ e_i$ (it is not a commutative polynomial by itself). Then $ e_ib_i(X_1, X_2, X_3) \in I$ for every $ i$, so we reduce to this case. Write $ b_i(X_1, X_2, X_3) $ as $ \sum_l c_{l}(X_1,X_2)X_3^l $. The $ \mathbb{G}_m$-invariance of $ I$ implies that $ \sum_l \lambda^l (e_ic_l(X_1,X_2)X_3^l) \in I $ for every $ \lambda \in \cx^*$. There are only finitely many distinct indices $ l$ in the sum, so by taking equally many distinct values of $ \lambda$ and using the invertibility of the Vandermonde matrix, we find that each term in the sum $ e_ic_l(X_1, X_2)X_3^l \in I$. In other words, $ I$ is invariant iff it is `monomial in the last variable'.
    \par
    The $ \mathbb{G}_m$-action on $ B$ which scales $ X_2 $ restricts to an action on the $ \mathbb{G}_m$-fixed points of the action above. So we can apply the method of the previous paragraph two more times, first scaling $ X_2$ and then scaling $ X_1$. At the end, the fixed points are exactly the monomial ideals. By Proposition 3.3, the Euler characteristic remains constant while passing to the $ \mathbb{G}_m$-fixed points at any step, so $ e(H^n_B) = \# \text{ of monomial ideals of codim n}$. 
    \par
    For a monomial ideal $ I$, let $ M$ be all the monomials in $ I$. Using the swapping relations, rewrite every monomial so that the idempotent occurs at the end. Write $ M$ as a disjoint union $ \bigsqcup_{i=0}^{c-1} M_i$ where $ M_i$ is the set of monomials ending in $ e_i$. Note that left-multiplication by elements of $ B $ maps $ M_i$ to itself. This is clear if we multiply on the left by an indeterminate $ X_i$. On the other hand, multiplying on the left by an idempotent of $ D$ either kills any monomial or leaves it invariant. Thus, each $ M_i $ determines a 3D-partition $ \pi_i$ of size $ |\pi_i|=n_i \ge 0$, and as a vector space, the monomials in $ \pi_0, \pi_1, \ldots, \pi_{c-1} $ together give a basis of $ B/I$ as a vector space. The converse is also true - given 3D-partitions $ \pi_0, \pi_1, \ldots , \pi_{c-1}$, reversing the construction yields a monomial left ideal of codimension $ \sum_i |\pi_i|$.
    \par
    It follows that the number of monomial ideals of $ B $ of codimension $ n$ is obtained by counting the number of 3D-partitions $ \pi_i $ such that $ \sum_{i=0}^{c-1} |\pi_i|= n $. For a fixed tuple $ (n_i) $, the count is $ \prod_{i=0}^{c-1} p_3(n_i) $, so we immediately get the stated expressions for $ e(H^n_B)$ and the generating series.
\end{proof}

\subsection{Local Models for the Moduli Spaces}

For the computations in this section, it will be convenient to take the definition of $ \mathscr{H}^n_{X,f}$ as parameterizing quotients $ \bo \to \mathcal{G} \to 0$ on $ X/f$ (see Lemma 2.13). Let  $ \mathfrak{c}_n : \mathscr{H}^n_{X,f} \to \cohn(X/f) \to \operatorname{Sym}^n(X/f) $ be the Hilbert-Chow morphism. 
\par
Let $ X/f \to \operatorname{Sym}^n X/f$ be the map $ P \to n[P]$ and let $ \mathscr{H}^n_{X,f}(n)$ be the inverse image $ \mathfrak{c}_n^{-1}(X/f) $. That is, if $[\bo \to \mathcal{G} \to 0] \in \mathscr{H}^n_{X,f}(n) $, the quotient $ \mathcal{G} $ is supported at a single point of $ X/f$, or equivalently, $ \mathcal{F}$ is supported in an orbit of $ f$ on $ X$.
\par
We first consider the two simple `local' cases in which the action may be free, or may have a 1-dimensional fixed locus.

\begin{lem} (The local case of a free action)
    Let $ X = \bigsqcup_{i=1}^d \mathbb{A}^2 $ and $ f : X \to X$ be an automorphism of order $ d$ that cyclically permutes the disjoint copies of $ \mathbb{A}^2$. The quotient $ X/f $ is $ \mathbb{A}^2 $. Then, there is a canonical isomorphism $ \mathscr{H}^n_{X,f}(n) \cong \mathfrak{c}_n^{-1}(n[O]) \times X/f$, where $ O$ is the origin in $ X/f = \mathbb{A}^2$. Consequently, $$ e(\mathfrak{c}_n^{-1} (n[O])) = \sum_{\sum_{i=0}^{d-1} n_i = n, n_i \ge 0} (\prod_{i=0}^{d-1} p_3(n_i)) $$
\end{lem}

\begin{proof}
    The group $ \mathbb{A}^2$ acts on $ X$ and $ X/f$ by translation, and this action commutes with the quotient map. So there is an action of $ \mathbb{A}^2$ on $ \mathscr{H}^n_{X,f}$, which is also compatible under the first two equivalent characterizations of Lemma 2.13. Using this action, a quotient supported at $ P \in X/f$ can be translated to a quotient supported at $ O$. Then, the natural map $ \mathscr{H}^n_{X,f}(n) \to \mathfrak{c}_n^{-1}(n[O]) \times X/f $ is easily seen to be an isomorphism.
    \par
    For the second part, since $ e(\hnxf(n)) = e(\mathfrak{c}_n^{-1} n[O]) $, it suffices to prove the statement for $ e(\hnxf(n)) $. In this local case, $ \bo$ is the sheaf of algebras associated to $ B = (\prod_{i=0}^{d-1} \cx)\langle x,y,z \rangle/(xy-yx, yz-zy, zx -xz)$ along with the relations $ e_{i+1}z = ze_i $. This description of $ B$ easily follows from the fact that $ B = (D[x,y])[z; \theta] $ where $ \theta $ maps $ e_i\cx[x,y] $ identically to $ e_{i+1} \cx[x,y] $.
    \par
    By Proposition 3.5, $ e(\hnxf)$ equals the right hand side of the equality above. But $ e(\hnxf(n)) = e(\hnxf)$, as $ \hnxf(n)$ is $ \mathbb{G}_m^3$-invariant (when scaling on $x,y,z $) and contains all the fixed points of the $ \mathbb{G}_m^3$-action on $ \hnxf$.
\end{proof}

\begin{lem} (The local case of a 1-dimensional fixed locus) 
    Let $ X = \mathbb{A}^2 = \operatorname{Spec} \cx[x,y]$ and $ f : X \to X$ be the automorphism described by $ f(x) = x, f(y) = \mu y $ for a primitive $ d$-th root of unity $ \mu$. The quotient is $ X/f = \mathbb{A}^2 = \operatorname{Spec} \cx[x,y^d] $ and the fixed locus $ R$ is the $ x$-axis, which may be identified in $ X$ and $ X/f$ under the quotient. Then, there is a canonical isomorphism $ \mathfrak{c}^{-1}_n(R) \cong \mathfrak{c}_n^{-1}(n[O]) \times R $. Consequently, $$ e(\mathfrak{c}_n^{-1} (n[O])) = p_3(n) $$
\end{lem}

\begin{proof}
    The proof of the first part is identical to that of Lemma 3.6 by applying translations only along the fixed locus. For the second part, $ \bo $ is the sheaf of algebras associated to the noncommutative polynomial ring $$ B = \mathbb{C} \langle x,y,z \rangle / (xy-yx, zy - \mu y z, zx-xz )$$ By Proposition 3.5, $ e(\hnxf) = p_3(n) $. The statement that $ e(\mathfrak{c}_n^{-1} (n[O])) $ equals $ p_3(n)$ follows the same way as in Lemma 3.6.
\end{proof}

Following the ideas in \cite{beentjes2021virtual} and \cite{BehrendFantechi2008}, we start with a more general setup. Let $ S$ be a quasi-projective variety with an automorphism $ g : S \to S $ of finite order and let $ S' $ be the quotient $ S/g$. Let $ \rho : T' \to S' $ be a morphism and $ T = T' \times_{S'} S $ be the base change. Let $ g_T : T \to T$ be the induced automorphism on $ T$. In addition, we will assume that there is an equivariant compactification $ (\overline{S}, \overline{g}) $ (see Lemma 2.14) of $ (S,g)$.
\par
If $ \rho $ is flat, then formation of the quotient commutes with base change, that is, $ T' = T/g_T$. Furthermore, the noncommutative algebras are also compatible: By flat base change, $ \rho^* \mathcal{B}_0(g) \cong \mathcal{B}_0(g_T)$ as $ \mathcal{O}_{T'}$-modules, and the noncommutative algebra structures on these two sheaves match because of the construction of $ T$ as $ T' \times_{S'} S$ (see Remark 2.8).

The following proposition is an analogue of Proposition A.2 in \cite{beentjes2021virtual} and Lemma 4.7 in \cite{BehrendFantechi2008}.
\newcommand{\bos}{\mathcal{B}_0(\sigma)}
\begin{prop}
    Let $ P = [\mathcal{B}_0(g_T) \to \mathcal{Q} \to 0]  \in \mathscr{H}^n_{T, g_T}$ be a point in the Hilbert scheme such that $ \rho $ is injective on $ |\operatorname{Supp} \mathcal{Q}|$ and assume $ \rho $ is etale in a neighborhood of $ |\operatorname{Supp} \mathcal{Q}| $. Then, there is an open subset $ P \in U \subset \mathscr{H}^n_{T, g_T} $ and a morphism $ U \to \mathscr{H}^n_{S,g}$ sending a quotient $ P' = [\mathcal{B}_0(g_T) \to \mathcal{Q}' \to 0] \in U $ to the composition \begin{equation}
        [\mathcal{B}_0(g) \to \rho_* \rho^* \mathcal{B}_0(g)_ = \rho_* \mathcal{B}_0(g_T) \to \rho_* \mathcal{Q}'] 
    \end{equation}
\end{prop}

\begin{proof}
    Begin with the case of $ S $ projective. We first show that (1) (which is a map of $ \mathcal{B}_0(g)$-modules) is surjective for the point $ P$. We can shrink $ T' $ to an open affine neighborhood $ V $ of $ |\operatorname{supp} \mathcal{Q}| $ where $ \rho $ is etale, and such that none of the points in $ V - |\operatorname{Supp} \mathcal{Q}| $ map to $ \rho(|\operatorname{Supp} \mathcal{Q}|)$. We can also replace the target $ S'$ by $ \rho(V)$, which is open as $ \rho$ is now etale. After these replacements, $ \rho$ is etale and surjective, so is faithfully flat. By faithful flatness for quasi-coherent sheaves, $ \mathcal{B}_0(g) \to \rho_* \mathcal{Q} $ is surjective iff $ \rho^* \mathcal{B}_0(g) \to \rho^* \rho_* \mathcal{Q} $ is. But by the choice of $ V$ as a neighborhood containing no other preimages of $ \rho(|\operatorname{Supp} \mathcal{Q}|) $, the canonical map $ \rho^* \rho_* \mathcal{Q} \to \mathcal{Q} $ is an isomorphism. Thus $ \rho^*\mathcal{B}_0(g) \to \rho^* \rho_* \mathcal{Q} $ is simply identified with the surjection $ \mathcal{B}_0(g_T) \to \mathcal{Q} $, proving the claim.
    \par To complete the proof that there is an open subset $ U $ around $ P$ and a morphism $  U \to \mathscr{H}^n_{S,g}$, it suffices to construct an open neighborhood of the point $ P\in \mathscr{H}^n_{T, g_T} $ such that for any quotient $ P' = [\mathcal{B}_0(g_T) \to \mathcal{Q}' \to 0]$ in this neighborhood, $ \mathcal{B}_0(g) \to \rho_* \mathcal{Q}' $ remains surjective. As mentioned above, reduce to $ T' = V $, and therefore, we can assume that $ \rho $ is affine. Consider the commutative diagram 
\[\begin{tikzcd}
	{T' \times \{P\}} & {T' \times \mathscr{H}^n_{T, g_T}} & T' \\
	{S' \times \{P\}} & {S' \times \mathscr{H}^n_{T, g_T}} & S'
	\arrow["i", from=1-1, to=1-2]
	\arrow["\rho"', from=1-1, to=2-1]
	\arrow["{p_{T'}}", from=1-2, to=1-3]
	\arrow["{\tilde{\rho}}"', from=1-2, to=2-2]
	\arrow["\rho", from=1-3, to=2-3]
	\arrow["j"', from=2-1, to=2-2]
	\arrow["{p_{S'}}"', from=2-2, to=2-3]
\end{tikzcd}\] In the right square, the base change map $ p_{S'}^* \rho_* \to \tilde{\rho}_* p_{T'}^*$ is an isomorphism because $ \rho $ is affine. On $ T' \times \mathscr{H}^n_{T, g_T}$, consider the universal family of quotients $$ p_{T'}^* \mathcal{B}_0(g_T) \to \mathscr{Q} \to 0 $$
From the canonical map $ \mathcal{B}_0(g) \to \rho_* \rho^* \mathcal{B}_0(g) = \rho_* \mathcal{B}_0(g_T) $ on $ S' $, we get  \begin{equation}
    p_{S'}^* \mathcal{B}_0(g) \to p_{S'}^* \rho_* \mathcal{B}_0(g_T) \cong \tilde{\rho}_* p_{T'}^* \mathcal{B}_0(g_T) \to \tilde{\rho}_* \mathscr{Q}
\end{equation}
     on $ S' \times \mathscr{H}^n_{T, g_T} $, which is precisely the universal version of the pointwise construction above. We know that (2) is a surjection when restricted to $ S' \times \{P\}$. Let $ \mathscr{K}$ be the cokernel of (2). Then $ \mathscr{K}$ is coherent as $ \tilde{\rho}_* \mathscr{Q} $ is (because $ \mathscr{Q}$ has proper support over $ \mathscr{H}^n_{T, g_T} $, so $ \operatorname{Supp} \mathscr{Q} \to S' \times \mathscr{H}^n_{T, g_T} $ is also proper). So $ |\operatorname{Supp} \mathscr{K}|$ is closed and its image under the projection $ S' \times \mathscr{H}^n_{T, g_T} \to \mathscr{H}^n_{T, g_T}$ is also closed, as $ S $, and hence $S'$, is projective. Let $ U$ be the open complement (containing $ P$ by construction), which we claim is our desired open subset. By construction, (2) is a surjection over $ S' \times U$. It defines a flat family of quotients: $ \tilde{\rho}_* \mathscr{Q} $ is flat over $ \mathscr{H}^n_{T, g_T}$ because $ \mathscr{Q}$ is, and $ \tilde{\rho} $ is affine. Thus, the desired morphism $ U \to \mathscr{H}^n_{S, g} $ is constructed by the restriction of (2) to $ S' \times U$.
     \par
     In the general case, we can replace $ (S,g)$ by $ (\overline{S}, \overline{g})$ and $ \rho $ by $ i\circ \rho$. The assumptions on $ P$ remain preserved, and the resulting morphism $ U \to \mathscr{H}^n_{\overline{S}, \overline{g}} $ factors through the open subscheme $ \mathscr{H}^n_{S,g}$.
\end{proof}

\begin{prop}
    Under the assumptions of Proposition 3.8, the morphism $ U \to \mathscr{H}^n_{S,g}$ so obtained is etale.
\end{prop}

\begin{proof}
     This is an application of the infinitesimal lifting criterion for etale morphisms. Consider a square of the form
\[\begin{tikzcd}
	{\operatorname{Spec}A} & U \\
	{\operatorname{Spec} A'} & \mathscr{H}^n_{S,g}
	\arrow[from=1-1, to=1-2]
	\arrow[from=1-1, to=2-1]
	\arrow[from=1-2, to=2-2]
	\arrow[dashed, from=2-1, to=1-2]
	\arrow[from=2-1, to=2-2]
\end{tikzcd}\] where $ \operatorname{Spec} A$ is an Artinian local $ \cx$-algebra whose closed point maps to $ P \in U$, and $ \operatorname{Spec} A \subset \operatorname{Spec} A'$ is a square-zero extension. We need to show that the dashed morphism $ \operatorname{Spec} A' \to U$ exists and is unique. 
\par
As in Proposition 3.8, we can assume $ T' =V $ so that $ \rho^* \rho_* \mathcal{Q} \to \mathcal{Q}$ is an isomorphism. This implies that $ \rho_{\operatorname{Spec}A}^* (\rho_{\operatorname{Spec}A})_* \mathcal{Q}_A \to \mathcal{Q}_A$ remains an isomorphism for any flat family of quotients $ \mathcal{Q}_A$ extending $ \mathcal{Q} $, on any fat point $ \operatorname{Spec} A$. 
\par 
Let $ \mathcal{T}$ be the family of quotients on $ S' \times \operatorname{Spec} A' $ defined by $ \operatorname{Spec} A' \to \mathscr{H}^n_{S,g}$. Then, tracing through the diagram and using the above observation easily implies that the family $ \rho_{\operatorname{Spec} A'}^* \mathcal{T}$ on $ T' \times \operatorname{Spec} A' $ provides the unique dashed arrow $ \operatorname{Spec} A' \to U$.
\end{proof}

\begin{cor}
    Let $ S,T,S', T', g, g_T$ and $ \rho$ be as before. Consider $ \mathscr{H}^n_{T, g_T}(n) $ and $ \mathscr{H}^n_{S,g}(n) $ with the reduced scheme structure. If $ \rho$ is etale, there is a Cartesian square
    \[\begin{tikzcd}
	\mathscr{H}^n_{T, g_T}(n) & \mathscr{H}^n_{S,g}(n) \\
	T' & S'
	\arrow[from=1-1, to=1-2]
	\arrow["{\mathfrak{c}_n}"', from=1-1, to=2-1]
	\arrow["{\mathfrak{c}_n}", from=1-2, to=2-2]
	\arrow["\rho"', from=2-1, to=2-2]
    \end{tikzcd}\]
    where the vertical maps are the respective Hilbert-Chow morphisms.
\end{cor}
\begin{proof}
    By Proposition 3.8 the top horizontal arrow exists and the square is commutative, while by Proposition 3.9, it is immediate that the square is also Cartesian.
\end{proof}

To reduce to the case of the local models, we will need versions of Luna's etale slice theorem for finite groups. The more general statements for reductive groups can be found in Luna's original paper \cite{Luna1973Slices}, or the comprehensive notes of Drézet \cite{drezet2004luna}.

\begin{prop}[Luna's etale slice theorem for finite groups] Let $ G$ be a finite group acting on an affine variety $ Z$. Let $ z$ be a point of $ Z$ for which the orbit $ Gz$ is closed. Then there is a locally closed affine subvariety $ V \subset Z$ such that $ z \in V$, and
\begin{enumerate}
    \item $ V$ is $ G_z$-invariant, where $ G_z \subset G$ is the stabilizer of $ z$.
    \item The induced $ G $-equivariant morphism $ G \times_{G_z} V \to Z$ is strongly etale.
\end{enumerate}
    
\end{prop} 

\begin{prop}[Linearization to the tangent space] Let $ G, Z$ and $ z$ satisfy the assumptions in Proposition 3.11. In addition, if $ Z$ is smooth at $ z$, then there is a locally closed subvariety $ z \in V \subset Z$ for which properties (1) and (2) above hold, and an etale $ G_z$-equivariant morphism $ \psi : V \to T_zZ  $ such that
\begin{enumerate}
    \item $ \psi(z) = 0 $ and $ d\psi_z : T_zV \to T_0(T_zV) = T_zV  $ is the identity.
    \item $ \psi$ is strongly etale.
\end{enumerate}
    
\end{prop}

Return to the case of $ X $ and $f$. Let $ W \subset X $ be the open subset where the action of $ f$ is free. We can give a direct description of the local model of $ \bo $ on $ W$ (it is a special case of Proposition 3.12). Let $ \rho $ be the quotient $ W \to W/f $. It is a Galois cover. Therefore, the square
\[\begin{tikzcd}
	W^d = \bigsqcup_{i=1}^{d} W & W \\
	W & {W/f}
	\arrow["p_2", from=1-1, to=1-2]
	\arrow["p_1"', from=1-1, to=2-1]
	\arrow["q|_W", from=1-2, to=2-2]
	\arrow["\rho"', from=2-1, to=2-2]
\end{tikzcd}\]
is Cartesian, where the top left corner is a disjoint union of $ d$ copies of $ W$. Note that the projections $ p_1$ and $ p_2$ are distinct - On the $ i $-th copy of $ W$ in $ W^d$, $ p_1$ is the identity while $ p_2$ is $ f^i $.
\par 
Let $ \sigma $ be the automorphism of $ W^d $ that cyclically permutes the disjoint copies: The $ i$-th copy of $ W$ maps to the $ (i+1)$-th copy. Then $ \sigma$ is the (unique) automorphism lifting $ f$.
\[\begin{tikzcd}
	{W^d} & {W^d} & \\
	W & W & W \\
	& {W/f}
	\arrow["\sigma", from=1-1, to=1-2]
	\arrow[from=1-1, to=2-1]
	\arrow[from=1-1, to=2-2]
	\arrow[from=1-2, to=2-1]
	\arrow[from=1-2, to=2-3]
	\arrow[from=2-1, to=3-2]
	\arrow["f", from=2-2, to=2-3]
	\arrow[from=2-2, to=3-2]
	\arrow[from=2-3, to=3-2]
\end{tikzcd}\] 
So this is a special case of the general setup, and hence we get $ \rho^* \bo \cong \bos$ as sheaves of noncommutative algebras.

\begin{cor}
    $ \mathscr{H}^n_{W,f}(n) \to W/f$ is a locally trivial fibration in the analytic topology.
\end{cor}

\begin{proof}
    By Corollary 3.10, there is a Cartesian square
    \[\begin{tikzcd}
	\mathscr{H}^n_{W^d, \sigma}(n) & \mathscr{H}^n_{W,f}(n) \\
	W & W/f
	\arrow[from=1-1, to=1-2]
	\arrow["{\mathfrak{c}_n}"', from=1-1, to=2-1]
	\arrow["{\mathfrak{c}_n}", from=1-2, to=2-2]
	\arrow["\rho"', from=2-1, to=2-2]
    \end{tikzcd}\] so it suffices to prove local triviality for $ \mathscr{H}^n_{W^d,\sigma}(n) \to W$ because $ \rho$ is etale. Cover $ W$ by Zariski open subsets $ W_i$ with etale maps $ W_i \to \mathbb{A}^2$. By Corollary 3.10 again, it is enough to prove the statement for $ (\mathbb{A}^2)^d = \sqcup_{i=1}^{d} \mathbb{A}^2 $ with a cyclic permutation automorphism. But this is true because of Lemma 3.6.
\end{proof}

\subsubsection{Cayley Quivers}
     As an interlude, we describe the quivers giving local models for sheaves with a twisted endomorphism $ (\f, \phi) $ such that the support of $ \f $ lies in a free orbit. 
     \par
     Let $ G$ be a finite abelian group. Let $ S = (f_1, \ldots, f_r) $ be a tuple of elements of $ G$ (It is allowed for elements to occur with multiplicity). 
    \par
    The Cayley quiver $ Q_{G,S} $ associated to $ (G,S) $ is the quiver with relations defined as follows: The vertex set of $ Q_{G,S}$ is $ G$ and for each generator $ f_i \in S$ and each $ g \in G$, a directed edge $ p^g_i $ is drawn from $ g$ to $ gf_i$. We impose the `obvious relations': For every pair of generators $ f_i,f_j$ and every $ g \in G$, we set $ p^{gf_i}_jp^g_i - p^{gf_j}_ip^g_j = 0$. Said more simply, the two paths $ g \to gf_i \to gf_if_j $ and $ g \to gf_j \to gf_jf_i $ are identical (the fact that $ G$ is abelian is used because the endpoints $ gf_if_j $ and $ gf_jf_i$ are equal). Let $ kQ_{G,S} $ be the path algebra of $ Q_{G,S} $ and $ I_{G,S} $ be the two-sided ideal defined by the set of relations above. The path algebra with relations is the quotient $ A_{G,S} = kQ_{G,S} / I_{G,S} $
    \par
    A rather simple description of the algebra $ A_{G,S}$ exists. Let $ \{ e_g \: | \: g \in G \} $ be the idempotent paths at the vertices of $ Q_{G,S} $. For each $ i $, let $ X_i = \sum_{g \in G} p^{g}_i $. Then $ X_ie_g = p^g_i = e_{gf_i}X_i $ for all $ i$ and $ g \in G$, and hence $ A_{G,S}$ is generated by the idempotents $ e_g$ and the sums $ X_i$. Furthermore, consider the product $ X_jX_i $. By definition, this product is the sum of paths $ g \to gf_i \to gf_if_j $ over all $ g \in G$. Similarly, $ X_iX_j$ is the sum of paths $ g \to gf_j \to gf_jf_i $ over all $ g \in G$. It follows that $ X_jX_i = X_iX_j $ because of the imposed relations. Also, for any $ a \in A_{G,S} $, $ a = \sum_{g \in G} a(g) $ where $ a(g) $ is the sum of all monomials in $ a$ which give paths that end at $ g$. If $ a(g) $ is a monomial, the path described by this monomial is of the form $$ g(f_{i_1}^{-1} \ldots f_{i_{m-1}}^{-1}f_{i_m}^{-1}) \to g(f_{i_1}^{-1} \ldots f_{i_m}^{-1}) \cdots \to g(f^{-1}_{i_1}) \to g $$ for some (possibly repeated) generators $ f_{i_1}, \ldots, f_{i_m} $ in $ S$. So $ a(g) = e_g(\lambda \cdot X_{i_1}\cdots X_{i_m}) $ for some scalar $ \lambda$, and in general, $ a(g) = e_ga_g(X)$ for a unique polynomial $ a_g(X) = a_g(X_{i_1}, \ldots, X_{i_m})$. Consider the subring of $ A_{G,S} $ generated by the idempotents $ e_g$. This is the ring $ D = \prod_{g \in G} ke_g $ which is a $ k$-algebra via the diagonal embedding $ k \hookrightarrow \prod_{g \in G} ke_g $ which maps $ \lambda \in k $ to $ (\lambda e_g)_{g \in G} $ (Recall that $ \sum_{g \in G} e_g = 1$ in the path algebra, so this is a well-defined algebra). Every generator $ f_i \in S $ describes a $ k$-algebra automorphism $ \theta_i : D \to D$, $ \theta_i(e_g) = e_{gf_i} $ (in fact, every element of $ G$ describes an automorphism). The collection of automorphisms $ \theta_i $ is pairwise commuting, since $ G$ is abelian. Therefore, $ A_{G,S} $ is isomorphic to the twisted polynomial ring $ D[X_1, \ldots, X_r; \theta_1, \ldots , \theta_r] $, where the `twisting action' on each $ X_i$ is via $ \theta_i $.
    \par
    Let $  G = \mathbb{Z}/d\mathbb{Z} \subset \operatorname{Aut}(X) $ and $ S = (f, id_X, id_X) $, Then, representations of the Cayley quiver of $ (G,S)$ exactly correspond to twisted endomorphisms $ \phi : \f \to f^* \f$ where $ |\operatorname{Supp} \f|$ is contained in a single orbit in $ W$. For example, the quiver for $ d=3$ is
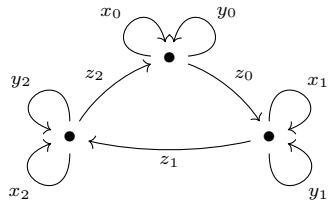
\[\begin{tikzcd}
	& \bullet \arrow[out=5,in=75,loop, swap, "y_0"] \arrow[out=175,in=105,loop, "x_0"] \\
	\bullet \arrow[out=90,in=160,loop, swap, "y_2"] \arrow[out=270,in=200,loop, "x_2"] && \bullet \arrow[out=90,in=20,loop, "x_1"] \arrow[out=270,in=340,loop, swap, "y_1"]
	\arrow["{z_0}", curve={height=-6pt}, from=1-2, to=2-3]
	\arrow["{z_2}", curve={height=-6pt}, from=2-1, to=1-2]
	\arrow["{z_1}", curve={height=-6pt}, from=2-3, to=2-1]
\end{tikzcd}\]
\vspace{6pt}
\par
Next, let $ R \subset X $ be any $ 1$-dimensional component in the fixed locus of $ f$. Let $ \hnxf(n)|_R$ be the stratum parameterizing stable triples $ (\f, \phi,s)$ with $ |\operatorname{Supp} \f| $ a single point in $ R$. We prove the analogue of Corollary 3.13 for $ R$.

\begin{cor}
    There is a locally trivial fibration $ \mathscr{H}^n_{X,f}(n)|_R \to R$ in the analytic topology.
\end{cor}

\begin{proof}
    Let $ p \in X$ be any point in $ R$. Applying Proposition 3.12, we get an affine open $ f$-invariant subset $ p \in V = \operatorname{Spec} A \subset X$, and a strongly etale morphism $ (\operatorname{Spec} A, p) \to (\mathbb{A}^2_{x,y}, O) $ such that $ f_V : V \to V$ commutes with the diagonal automorphism $ g(x)= x, g(y)= \mu y$ for a primitive $ d$-th root of unity $ \mu$. So $ V/f_V \subset X/f$ is smooth and $ R \cap V$ is isomorphically mapped to its image in $ V/f_V $. So $ R$ is identified with its image in $ X/f $ and the map $ \mathscr{H}^n_{X,f}(n)|_R \to R$ is part of the Hilbert-Chow map. It suffices to prove the triviality locally around p.
    \par
    The chart $ \operatorname{Spec} A \to \mathbb{A}^2_{x,y}$ is strongly etale. This means that the square 
    \[\begin{tikzcd}
	V & \mathbb{A}^2_{x,y} \\
	{V/f_V} & {\mathbb{A}^2_{x,y^d}}
	\arrow[from=1-1, to=1-2]
	\arrow[from=1-1, to=2-1]
	\arrow[from=1-2, to=2-2]
	\arrow[from=2-1, to=2-2]
\end{tikzcd}\] is Cartesian, and the induced map between the quotients $ V/f_V \to \mathbb{A}^2_{x, y^d}$ is also etale.
    By Corollary 3.10, there is a Cartesian square
    \[\begin{tikzcd}
	\mathscr{H}^n_{V, f_V}(n) & \mathscr{H}^n_{\mathbb{A}^2, g}(n) \\
	{V/f_V} & {\mathbb{A}^2_{x,y^d}}
	\arrow[from=1-1, to=1-2]
	\arrow["{\mathfrak{c}_n}"', from=1-1, to=2-1]
	\arrow["{\mathfrak{c}_n}", from=1-2, to=2-2]
	\arrow[from=2-1, to=2-2]
\end{tikzcd}\] The fixed locus $ R \cap V $ mapping to the $ x $-axis remains an etale map. By Lemma 3.7, the right vertical map is a trivial fibration over the $ x$-axis, so the left vertical map is a locally trivial fibration over $ R \cap V$, completing the proof.
\end{proof}

\subsection{Proof of Theorem 1.2} 

\subsubsection{Proof in the special case of prime order automorphisms}For simplicity and to illustrate the computation, we first prove the theorem in the case when $ \text{ord} f= d$ is prime. Divide the surface into a disjoint union: $ X = W \sqcup R^1 $, where $ W$ is the (open) locus on which the action of $ f$ is free, and $ R^1$ is the fixed locus of $ f$. Further, $ R^1 $ is a disjoint union $ R^1_0 \sqcup R^1_1 $, where $ R^1_0$ is the union of all isolated fixed points of $ f$, and $ R^1_1$ is the union of the 1-dimensional components in the fixed locus of $ f$. Therefore, $ X/f $ is the union of the strata $ W/f \sqcup R^1_0 \sqcup R^1_1 $ where we have identified $ R^1_0$ and $ R^1_1$ with their images, as closed subsets. 
\par
From now on, let $ \hnxf$ be $ Z^n$, and for any strata $ A$ in $ X/f$, denote by $ Z^n_A$ the stratum of $ Z^n$ parameterizing quotients whose topological support is contained in $ A$. As locally closed subsets, $ Z^n = \bigsqcup_{\substack{a,b,c \ge0 \\ a+b+c = n}} Z^a_{W/f} \times Z^b_{R^0} \times Z^c_{R^1} $ and this is enough to calculate topological Euler characteristics, so we have \[ \sum_{n \ge 0} e(Z^n)q^n = \left(\sum_{a \ge 0} e(Z^a_{W/f})q^a \right) \left( \sum_{b \ge 0} e(Z^b_{R^0})q^b\right) \left(\sum_{c \ge 0} e(Z^c_{R^1}) q^c\right) \]
We can compute each of the three series on the right hand side. Consider the series $\sum_{a \ge 0} e(Z^a_{W/f})q^a $. For a fixed but arbitrary $ a$, let $ \alpha = (\alpha_1, \ldots, \alpha_r) $ be a partition of $ a$ of length $ l(\alpha) = r $. Let $ Z^{a, \alpha}_{W/f}$ be the stratum parameterizing quotients whose support has $ r$ distinct points of multiplicities $ \alpha_1, \ldots, \alpha_r$. The open subscheme $ U \subset \prod_i Z^{\alpha_i, (\alpha_i)}_{W/f}$ consisting of quotients whose supports are pairwise distinct points has an obvious morphism to $ Z^{a, \alpha}_{W/f}$. It is a Galois cover with automorphism group $ G_{\alpha} $ equal to the automorphisms of the partition $ \alpha$ (A direct product of symmetric groups). Note that $ U$ can be described by the Cartesian square
\[\begin{tikzcd}
	U & \prod_i Z^{\alpha_i, (\alpha_i)}_{W/f} \\
	{(W/f)_0^r} & {(W/f)^r}
	\arrow[from=1-1, to=1-2]
	\arrow[from=1-1, to=2-1]
	\arrow[from=1-2, to=2-2]
	\arrow[from=2-1, to=2-2]
\end{tikzcd}\] where the right vertical morphism takes quotients to their support points and $ (W/f)^r_0$ is the open subset consisting of pairwise distinct tuples. Therefore \[ e(Z^{a, \alpha}_{W/f}) = \frac{1}{|G_{\alpha}|} e(U) = \frac{1}{|G_{\alpha}|} e((W/f^r)_0) \prod_{i=1}^r j(\alpha_i) \] where $ j(\alpha_i) = \sum_{\sum_{k=0}^{d-1} n_k = \alpha_i, n_k \ge 0} (\prod_{i=0}^{d-1} p_3(n_k)) $ and the last equality holds due to Lemma 3.6 and Corollary 3.13. Since $ (W/f)^r_0 \to (W/f)^{r-1}_0 $ is a locally trivial fiber bundle, it is easy to see by induction, that $ e((W/f)^r_0) = r! \binom{e(W/f)}{r}$ for all $ r \ge 1$. So,
\begin{align*}
    \sum_{a \ge 0} e(Z^a_{W/f})q^a &= 1 + \sum_{a \ge 1} \sum_{\alpha \vdash a} e(Z^{a, \alpha}_{W/f})q^{|\alpha|} \\
     &= 1 + \sum_{r \ge 1} \sum_{\alpha; \: l(\alpha)=r} \frac{1}{|G_{\alpha}|} e((W/f)^r_0) \prod_{i=1}^r (j(\alpha_i)q^{\alpha_i} ) \\
     &= 1 + \sum_{r \ge 1} \binom{e(W/f)}{r} \sum_{\alpha; \: l(\alpha)=r} \frac{r!}{|G_{\alpha}|} \prod_{i=1}^r (j(\alpha_i)q^{\alpha_i} )
\end{align*}
Let $ J(q) = \sum_{m \ge 1} j(m)q^m $. The inner sum $ \sum_{\alpha; \: l(\alpha)=r} \frac{r!}{|G_{\alpha}|} \prod_{i=1}^r (j(\alpha_i)q^{\alpha_i} ) $ is simply the product $ \prod_{i=1}^r J(q) = J(q)^r$ by multinomial expansion. So we obtain
\begin{align*}
    \sum_{a \ge 0} e(Z^a_{W/f})q^a &= 1 + \sum_{r \ge 1} \binom{e(W/f)}{r} J(q)^r \\
     &= (1 + J(q))^{e(W/f)} \\
     &= (\sum_{m \ge 0} j(m)q^m)^{e(W/f)} \\
     &= (M(q)^{d})^{e(W/f)} \\
     &= M(q)^{e(W)}
\end{align*}
where we used Proposition 3.5 in the second-last equality.
\par
Similarly, working with every connected component of $ R^1_1$ and applying the same steps using Lemma 3.7 and Corollary 3.14 instead, we get $ \sum_{c \ge 0} e(Z^c_{R^1_1}) q^c = M(q)^{e(R^1_1)} $.
\par
Lastly, let $ p \in X$ be an isolated fixed point. The linearization in Corollary 3.14 applies the same way in this case. We find an affine open $ f$-invariant subset $ p \in V = \operatorname{Spec}A$ and a strongly etale map $ \operatorname{Spec} A \to \mathbb{A}^2_{x,y}$ such that $ f_V$ commutes with the automorphism $ g(x) = \mu_1 x, g(y) = \mu_2 y$ for some primitive $ d$-th roots of unity $ \mu_1, \mu_2$. Then $ \mathcal{B}_0(g)$ is the sheaf associated to the noncommutative algebra \[ B = \cx \langle x,y,z \rangle/ (xy-yx, zx - \mu_1 xz, zy - \mu_2 yz ) \] and $ Z^b_{ \{ p \}} $ is isomorphic to the stratum of the Hilbert scheme $ H^b_{B}$ consisting of quotients supported at the origin. As in Lemma 3.7, the Euler characteristic of this stratum is $ p_3(b) $, so $ e(Z^b_{ \{ p \}}) = p_3(b) $ for any isolated fixed point $ p$. So \[ \sum_{b \ge 0} e(Z^b_{R^0})q^b = \left(\sum_{m \ge 0} p_3(m)q^m \right)^{\# \text{ of isolated fixed points}} = M(q)^{e(R^1_0)} \]
Putting it all together, \[ \sum_{n \ge 0} e(Z^n)q^n = M(q)^{e(W) + e(R^1_0) + e(R^1_1) } = M(q)^{e(W) + e(R^1)} = M(q)^{e(X)} \]

\subsubsection{Proof in the general case} In the general case, the only additional complication is keeping track of all possible stabilizers. Suppose $ b $ divides $ d$ and write $ d=bc $. Let $ R^c $ be the locally closed subset whose points have stabilizer group $ \mathbb{Z}/b\mathbb{Z}$ (so that if $ b=d$, $ R^c = R^1$ as defined in the special case of prime order). In view of the special case, it suffices to show that the `contributions coming from $ R^c$' to the series $ \sum_{n \ge 0} e(Z^n)q^n $ is $ M(q)^{e(R^c)} $. For this, the computation above shows that it is sufficient to prove the analogues of Corollary 3.13 and 3.14. We proceed to show this.
\par
Let $ p \in R^c$, so that the subgroup generated by $ f^c$ fixes $ p$. By Proposition 3.11, we find an affine open $ f^c$-invariant neighborhood $ V $ of $ p$ and a strongly etale $ \mathbb{Z}/d \mathbb{Z} $-equivariant morphism $ \bigsqcup_{i=0}^{c-1} V \to X $. It is easy to check that on the left, the generator $ f \in \mathbb{Z}/d \mathbb{Z} $ acts by mapping the $i$-th copy of $ V$ to the $ (i+1)$-th copy by the identity for all $ i = 0, 1, \ldots, c-2$, and the last copy to the $ 0$-th copy by $ f^c$. Therefore, there is a Cartesian square
\[\begin{tikzcd}
	\bigsqcup_{i=0}^{c-1} V  & X \\
	{V/f^c} & {X/f}
	\arrow[from=1-1, to=1-2]
	\arrow[from=1-1, to=2-1]
	\arrow[from=1-2, to=2-2]
	\arrow[from=2-1, to=2-2]
\end{tikzcd}\]
in which the top map sends the $ i $-th copy to $ V \subset X \xrightarrow{f^i} X $. By Proposition 3.12, we can also assume that there is a strongly etale $ \mathbb{Z}/b{Z}$-equivariant morphism $ (V,p) \to (\mathbb{A}^2_{x,y},0)$ where the action on the right scales $ x \to \mu_1 x $ and $ y \to \mu_2 y$ for some $ b$-th roots of unity $ \mu_1, \mu_2$ (since the representation is faithful, the $ \operatorname{lcm} $ of the orders of $ \mu_1$ and $ \mu_2$ is $ b$). Using Proposition 3.8 repeatedly, to show the triviality of $ \mathscr{H}^n_{X,f}(n)|_R \to R/f $, we are reduced to proving it for $ X = \bigsqcup_{i=0}^{c-1} \mathbb{A}^2_{x,y} $ and $f : X \to X $ induced by $ (V,p) \to (\mathbb{A}^2_{x,y},0) $.
\par
For this local case, $ X = \operatorname{Spec}(\prod_{i=0}^{c-1} ke_i)[x,y] = \operatorname{Spec} A$. The automorphism $ \theta: A \to A$ acts on $ x$ as \[ \theta(x) = \theta(e_0x + \ldots + e_{c-1}x) = e_1x + \ldots + e_{c-1}x + e_0 \mu_1 x = (1+ (\mu_1-1)e_0)x \] and similarly, $ \theta(y) = (1+ (\mu_2-1)e_0)y $. The coefficients $ \alpha_x = 1+(\mu_1-1)e_0 $ and $ \alpha_y = 1+(\mu_2-1)e_0 $ are units because $ (1+(\mu_1-1)e_0)^b = 1+(\mu_1^b-1)e_0 = 1$. The sheaf $ \bo$ is associated to the noncommutative algebra \[ B = ( \prod_{i=0}^{c-1} ke_i )[x,y, z]/(xy-yx, zx-\alpha_xxz, zy-\alpha_yyz) \] with the relations $ e_{i+1}z = ze_i $. By Proposition 3.5, \[ e(H^n_{B}) = j(n) = \sum_{\sum_{i=0}^{c-1} n_i = n, n_i \ge 0} (\prod_{i=0}^{c-1} p_3(n_i))\] We note that $ e(R^c/f) = \frac{e(R^c)}{c} $, as $ f^c$ stabilizes $ R^c$ and $ R^c \to R^c/f$ is etale of degree $ c$. Therefore, the series $ (\sum_{m \ge 0} j(m)q^m)^{e(R^c/f)}$ simplifies to $ M(q)^{e(R^c)}$, completing the proof in the general case.

\begin{cor}
    Let $ Y $ and $ \af $ be the compactification as in 2.2.1. Then the series $ \sum_{n \ge 0} e(\operatorname{Hilb}^n(Y, \af))q^n$ equals $ M(q)^{2e(X)} $. 
\end{cor}
\begin{proof}
    In $ Y $, let $ D_0 = X/f \times \{0\}$ and $ D_{\infty} = X/f \times \{ \infty \} $ be the divisors at $ 0$ and $ \infty$, and let $ Y^{\times} = X/f \times \mathbb{G}^m$. As before, we can write $ \operatorname{Hilb}^n(Y, \af) $ as $\bigsqcup_{\substack{a,b,c \ge0 \\ a+b+c = n}} Z^a_{D_0} \times Z^b_{D_{\infty}} \times Z^c_{Y^{\times}} $ where $ Z^m_A$ denotes the stratum parametrizing quotients of length $ m$ whose topological support is in $ A$. By Remark 2.19, $ e(Z^c_{Y^{\times}}) = 0$ for all $ c \ge 1$. Moreover, $ e(Z^a_{D_0}) = e(\mathscr{H}^a_{X,f})$ and $ e(Z^b_{D_{\infty}}) = e(\mathscr{H}^b_{X,f^{-1}}) $ which are immediate consequences of Remark 2.19 applied to the decomposition $$ \mathscr{H}^a_{X,f} = \bigsqcup_{a'+a'' = a} Z^{a'}_{D_0} \times Z^{a''}_{Y^{\times}}$$ for example, as locally closed subsets. So \[ \sum_{n \ge 0} e(\operatorname{Hilb}^n(Y, \af))q^n = \left(\sum_{a \ge 0} e(\mathscr{H}^a_{X,f})q^a \right) \left(\sum_{b \ge 0} e(\mathscr{H}^b_{X,f^{-1}})q^b \right) = M(q)^{2e(X)} \]
\end{proof}
\vspace{6pt}
\par
Since all fibrations are locally trivial in the analytic topology, the equality in Theorem 1.2 does not seem to have an obvious motivic refinement (except in the case $ f = id_X$), which is in contrast to the motivic identities in \cite{gottsche2001motive} and \cite{GuseinZade2004PowerStructure}. On the other hand, the degree $0 $ Donaldson-Thomas invariants (see \cite{Maulik2006MNOP2}) are defined for certain surfaces $ X $ and $ f = id_X$, as in that case, the compactification $ \operatorname{Hilb}^n(Y, \af)$ is the Hilbert scheme $ \operatorname{Hilb}^n(X \times \mathbb{P}^1) $. In a future paper, we will define similar invariants in some cases when $ f$ is non-trivial. This will extend the existing similarity in the enumerative theory for threefolds and the enumerative theory for surfaces with an automorphism.

\bibliographystyle{plain} % Defines the layout style
\bibliography{references}

\begin{thebibliography}{10}

\bibitem{beauville1989spectral}
Arnaud Beauville, Mudumbai~S. Narasimhan, and Sundararaman Ramanan.
\newblock Spectral curves and the generalized theta divisor.
\newblock {\em Journal f{\"u}r die reine und angewandte Mathematik (Crelles Journal)}, 1989(398):169--179, 1989.

\bibitem{beentjes2021virtual}
Sjoerd~Viktor Beentjes and Andrea~T. Ricolfi.
\newblock Virtual counts on quot schemes and the higher rank local {DT}/{PT} correspondence.
\newblock {\em Mathematical Research Letters}, 28(4):967--1032, 2021.

\bibitem{BehrendFantechi2008}
Kai Behrend and Barbara Fantechi.
\newblock Symmetric obstruction theories and hilbert schemes of points on threefolds.
\newblock {\em Algebra \& Number Theory}, 2(3):313--345, 2008.

\bibitem{BialynickiBirula1973OnFP}
Andrzej Bia{\l}ynicki-Birula.
\newblock On fixed point schemes of actions of multiplicative and additive groups.
\newblock {\em Topology}, 12(2):99--103, 1973.

\bibitem{bryan2022euler}
Jim Bryan and {\'A}d{\'a}m Gyenge.
\newblock Euler characteristics of hilbert schemes of points on surfaces with simple singularities.
\newblock {\em {\'E}pijournal de G{\'e}om{\'e}trie Alg{\'e}brique}, 6, 2022.

\bibitem{cheah96}
Jan Cheah.
\newblock On the cohomology of {Hilbert} schemes of points.
\newblock {\em Journal of Algebraic Geometry}, 5:479--512, 1996.

\bibitem{drezet2004luna}
Jean-Marc Dr{\'e}zet.
\newblock Luna's slice theorem and applications.
\newblock In Jaros{\l}aw~A. Wi{\'s}niewski, editor, {\em Algebraic group actions and quotients}, pages 39--90. Hindawi Publishing Corporation, 2004.

\bibitem{fantechi2024stack0dimensionalcoherentsheaves}
Barbara Fantechi and Andrea~T. Ricolfi.
\newblock On the stack of 0-dimensional coherent sheaves: structural aspects, 2024.

\bibitem{fantechi2025stack0dimensionalcoherentsheaves}
Barbara Fantechi and Andrea~T. Ricolfi.
\newblock On the stack of 0-dimensional coherent sheaves: motivic aspects, 2025.

\bibitem{fasola2021higher}
Nadir Fasola, Sergej Monavari, and Andrea~T. Ricolfi.
\newblock Higher rank {K}-theoretic {D}onaldson--{T}homas theory of points.
\newblock {\em Forum of Mathematics, Sigma}, 9:e15, 2021.

\bibitem{gottsche2001motive}
Lothar G{\"o}ttsche.
\newblock On the motive of the hilbert scheme of points on a surface.
\newblock {\em Mathematical Research Letters}, 8(5):613--627, 2001.

\bibitem{EGAIII_2}
Alexander Grothendieck and Jean Dieudonn{\'e}.
\newblock {\'E}l{\'e}ments de g{\'e}om{\'e}trie alg{\'e}brique: {III}. {{\'E}}tude cohomologique des faisceaux coh{\'e}rents, {S}econde partie.
\newblock {\em Publications Math{\'e}matiques de l'IH{\'E}S}, 17:5--91, 1963.

\bibitem{GuseinZade2004PowerStructure}
Sabir~M. Gusein-Zade, Ignacio Luengo, and Alejandro Melle-Hern{\'a}ndez.
\newblock A power structure over the {G}rothendieck ring of varieties.
\newblock {\em Mathematical Research Letters}, 11(1):101--107, 2004.

\bibitem{gyenge2018euler}
{\'A}d{\'a}m Gyenge, Andr{\'a}s N{\'e}methi, and Bal{\'a}zs Szendr{\H{o}}i.
\newblock Euler characteristics of {H}ilbert schemes of points on simple surface singularities.
\newblock {\em European Journal of Mathematics}, 4(1):439--524, 2018.

\bibitem{Hartshorne1966}
Robin Hartshorne.
\newblock Connectedness of the hilbert scheme.
\newblock {\em Publications Math{\'e}matiques de l'IH{\'E}S}, 29:5--48, 1966.

\bibitem{huybrechts2010geometry}
Daniel Huybrechts and Manfred Lehn.
\newblock {\em The Geometry of Moduli Spaces of Sheaves}.
\newblock Cambridge Mathematical Library. Cambridge University Press, 2nd edition, 2010.

\bibitem{larsen2012hilbertschemespointsassociative}
Michael Larsen and Valery~A. Lunts.
\newblock Hilbert schemes of points for associative algebras, 2012.

\bibitem{levine2009algebraic}
Marc Levine and Rahul Pandharipande.
\newblock Algebraic cobordism revisited.
\newblock {\em Inventiones mathematicae}, 176(2):243--286, 2009.

\bibitem{Li2006Zero}
Jun Li.
\newblock Zero dimensional {D}onaldson--{T}homas invariants of threefolds.
\newblock {\em Geometry \& Topology}, 10(4):2117--2171, 2006.

\bibitem{lieblich2006remarksstackcoherentalgebras}
Max Lieblich.
\newblock Remarks on the stack of coherent algebras, 2006.

\bibitem{Luna1973Slices}
Domingo Luna.
\newblock Slices {\'e}tales.
\newblock In {\em Sur les groupes alg{\'e}briques}, number~33 in Bulletin de la Soci{\'e}t{\'e} math{\'e}matique de France. M{\'e}moire, pages 81--105. Soci{\'e}t{\'e} math{\'e}matique de France, 1973.

\bibitem{Maulik2006MNOP2}
Davesh Maulik, Nikita Nekrasov, Andrei Okounkov, and Rahul Pandharipande.
\newblock Gromov--{W}itten theory and {D}onaldson--{T}homas theory, {II}.
\newblock {\em Compositio Mathematica}, 142(5):1286--1304, 2006.

\bibitem{pietromonaco2021g}
Stephen Pietromonaco.
\newblock G-invariant {H}ilbert schemes on {A}belian surfaces and enumerative geometry of the orbifold {K}ummer surface.
\newblock {\em Research in the Mathematical Sciences}, 8(4):56, 2021.

\bibitem{simpson1994moduli}
Carlos~T. Simpson.
\newblock Moduli of representations of the fundamental group of a smooth projective variety. {I}.
\newblock {\em Publications Math{\'e}matiques de l'IH{\'E}S}, 79:47--129, 1994.

\bibitem{stacks-project}
The {Stacks project authors}.
\newblock The stacks project.
\newblock \url{https://stacks.math.columbia.edu}, 2026.

\end{thebibliography}
\end{document}